\documentclass[reqno,11pt]{article}
\usepackage{mathrsfs}
\usepackage{amssymb,amsmath}

\UseRawInputEncoding
\usepackage{amsthm}
\usepackage{amsmath}
\usepackage{amsfonts}
\usepackage{enumerate}
\usepackage{color}
\usepackage{bm}

\usepackage{graphicx}

\numberwithin{equation}{section}
\newtheorem{theorem}{Theorem}[section]
\newtheorem{lemma}[theorem]{Lemma}
\newtheorem{corollary}[theorem]{Corollary}
\newtheorem{proposition}[theorem]{Proposition}

\theoremstyle{definition}
\newtheorem{definition}[theorem]{Definition}
\newtheorem{assumption}[theorem]{Assumption}
\newtheorem{example}[theorem]{Example}

\theoremstyle{remark}
\newtheorem{remark}[theorem]{Remark}

\begin{document}
\title{Smooth Subsonic and Transonic Flows with Nonzero Angular Velocity and Vorticity to steady Euler-Poisson system in a Concentric Cylinder}
\author{Shangkun Weng\thanks{School of Mathematics and Statistics, Wuhan University, Wuhan, Hubei Province, 430072, People's Republic of China. Email: skweng@whu.edu.cn}\and Wengang Yang\thanks{School of Mathematics and Statistics, Wuhan University, Wuhan, Hubei Province, 430072, People's Republic of China. Email: yangwg@whu.edu.cn}\and Na Zhang\thanks{School of Mathematics and Statistics, Wuhan University, Wuhan, Hubei Province, 430072, People's Republic of China. Email: nzhang@whu.edu.cn}}
\date{}
\maketitle

\def\be{\begin{eqnarray}}
\def\ee{\end{eqnarray}}
\def\be{\begin{equation}}
\def\ee{\end{equation}}
\def\ba{\begin{aligned}}
\def\ea{\end{aligned}}
\def\bay{\begin{array}}
\def\eay{\end{array}}
\def\bca{\begin{cases}}
\def\eca{\end{cases}}
\def\p{\partial}
\def\hphi{\hat{\phi}}
\def\bphi{\bar{\phi}}
\def\no{\nonumber}
\def\eps{\epsilon}
\def\de{\delta}
\def\De{\Delta}
\def\om{\omega}
\def\Om{\Omega}
\def\f{\frac}
\def\th{\theta}
\def\vth{\vartheta}
\def\la{\lambda}
\def\lab{\label}
\def\b{\bigg}
\def\var{\varphi}
\def\na{\nabla}
\def\ka{\kappa}
\def\al{\alpha}
\def\La{\Lambda}
\def\ga{\gamma}
\def\Ga{\Gamma}
\def\ti{\tilde}
\def\wti{\widetilde}
\def\wh{\widehat}
\def\ol{\overline}
\def\ul{\underline}
\def\Th{\Theta}
\def\si{\sigma}
\def\Si{\Sigma}
\def\oo{\infty}
\def\q{\quad}
\def\z{\zeta}
\def\co{\coloneqq}
\def\eqq{\eqqcolon}
\def\di{\displaystyle}
\def\bt{\begin{theorem}}
\def\et{\end{theorem}}
\def\bc{\begin{corollary}}
\def\ec{\end{corollary}}
\def\bl{\begin{lemma}}
\def\el{\end{lemma}}
\def\bp{\begin{proposition}}
\def\ep{\end{proposition}}
\def\br{\begin{remark}}
\def\er{\end{remark}}
\def\bd{\begin{definition}}
\def\ed{\end{definition}}
\def\bpf{\begin{proof}}
\def\epf{\end{proof}}
\def\bex{\begin{example}}
\def\eex{\end{example}}
\def\bq{\begin{question}}
\def\eq{\end{question}}
\def\bas{\begin{assumption}}
\def\eas{\end{assumption}}
\def\ber{\begin{exercise}}
\def\eer{\end{exercise}}
\def\mb{\mathbb}
\def\mbR{\mb{R}}
\def\mbZ{\mb{Z}}
\def\mc{\mathcal}
\def\mcS{\mc{S}}
\def\ms{\mathscr}
\def\lan{\langle}
\def\ran{\rangle}
\def\lb{\llbracket}
\def\rb{\rrbracket}
\def\fr#1#2{{\frac{#1}{#2}}}
\def\dfr#1#2{{\dfrac{#1}{#2}}}
\def\u{{\textbf u}}
\def\v{{\textbf v}}
\def\w{{\textbf w}}
\def\d{{\textbf d}}
\def\nn{{\textbf n}}
\def\x{{\textbf x}}
\def\e{{\textbf e}}
\def\D{{\textbf D}}
\def\U{{\textbf U}}
\def\M{{\textbf M}}
\def\F{{\mathcal F}}
\def\I{{\mathcal I}}
\def\W{{\mathcal W}}
\def\div{{\rm div\,}}
\def\curl{{\rm curl\,}}
\def\R{{\mathbb R}}
\def\g{{\textbf g}}
\def\FF{{\textbf F}}
\def\A{{\textbf A}}
\def\R{{\textbf R}}
\def\r{{\textbf r}}
\def\mE{\mathcal{E}}
\def\Div{{\rm div}}
\def\e{\mathfrak{e}}
\def\i{\mathfrak{i}}

\begin{abstract}
  In this paper, both smooth subsonic and transonic flows to steady Euler-Poisson system in a concentric cylinder are studied. We first establish the existence of cylindrically symmetric smooth subsonic and transonic flows to steady Euler-Poisson system in a concentric cylinder. On one hand, we investigate the structural stability of smooth cylindrically symmetric subsonic flows under three-dimensional perturbations on the inner and outer cylinders. On the other hand, the structural stability of smooth transonic flows under the axi-symmetric perturbations are examined. There is no any restrictions on the background subsonic and transonic solutions. A deformation-curl-Poisson decomposition to the steady Euler-Poisson system is utilized in our work to deal with the hyperbolic-elliptic mixed structure in subsonic region. It should be emphasized that there is a special structure of the steady Euler-Poisson system which yields a priori estimates and uniqueness of a second order elliptic system for the velocity potential and the electrostatic potential.
\end{abstract}

\begin{center}
\begin{minipage}{5.5in}
Mathematics Subject Classifications 2020: 35L65, 35M10, 76N10, 76N15.\\
Key words: Euler-Poisson system, structural stability, transonic flows, subsonic flows, deformation-curl-Poisson decomposition.
\end{minipage}
\end{center}
\section{Introduction and main results} \noindent

This paper concerns the steady inviscid compressible Euler-Poisson flows in a concentric cylinder $\Omega=\{(x_1, x_2, x_3): r_0< r=\sqrt{x_1^2+x_2^2}< r_1, -1< x_3< 1\}$. The flow is governed by the following equations
\be\label{EPS}
\begin{cases}
\p_{x_1}(\rho u_1)+\p_{x_2}(\rho u_2)+\p_{x_3}(\rho u_3)=0,\\
\p_{x_1}(\rho u_1^2)+\p_{x_2}(\rho u_1u_2)+\p_{x_3}(\rho u_1u_3)+\p_{x_1}P=\rho\cdot\p_{x_1}\Phi,\\
\p_{x_1}(\rho u_1u_2)+\p_{x_2}(\rho u_2^2)+\p_{x_3}(\rho u_2u_3)+\p_{x_2}P=\rho\cdot\p_{x_2}\Phi,\\
\p_{x_1}(\rho u_1u_3)+\p_{x_2}(\rho u_2u_3)+\p_{x_3}(\rho u_3^2)+\p_{x_3}P=\rho\cdot\p_{x_3}\Phi,\\
\sum\limits_{j=1}^3\p_{x_j}(\rho(\frac1 2|\u|^2+e)\u+P\u)=\rho\u\cdot\nabla\Phi,\\
\Delta\Phi=\rho-b(x),
\end{cases}
\ee
where $\rho$, $\u$, $P$, $e$ stand for the density, velocity, pressure and the internal energy respectively. $\Phi$ represents the eletrostatic potential and $E=\nabla\Phi$ is the electric field, which is generated by the Coulomb force of particles, and $b>0$ denotes the density of fixed, positively charged background ions. In addition, we introduce the equation of state
\be
P=P(\rho, e)=(\ga-1)\rho e=A(s)\rho^{\ga},
\ee
where $\ga>1$ is called the $adiabatic\,\, constant$,
and $A(S)=ae^S$, a is positive constant. In addition, the Bernoulli's function is given by
\be\label{B}
B=\frac1 2|{\bf u}|^2+\frac{\ga}{\ga-1}\frac{P}{\rho}.\no
\ee

As one of the essential equations in modeling physical flows, the problems for Euler-Poisson equations have been extensively studied, among which the significant progress on the transonic shock solutions of Euler-Poisson system has been made by Luo and Xin in \cite{LuoXin}. They gave a thorough study and presented a precise classification of the one-dimensional transonic shock solutions for both cases when $b$ is in supersonic and subsonic regions. The existence, non-existence, uniqueness, and non-uniqueness of solutions with transonic shocks are obtained according to the different cases of boundary data and physical interval length. And in terms of the solutions obtained in \cite{LuoXin}, the authors in \cite{Xie2011} investigated the stability, including structural and dynamical stabilities, provided the background electric field is not too negative at the shock location.

The structural stability of subsonic flows for the Euler-Poisson system has also motivated many researches in recent years. In \cite{Weng6}, the author proved the uniquely existence of three dimensional subsonic flows with nonzero vorticity in a rectangular nozzle under some smallness assumptions on the background subsonic state. In \cite{Bae1}, the existence and stability of subsonic flows for the steady Euler-Poisson system were established for potential flows in multidimensional nozzles by finding a coercive estimate for the linearized second order elliptic system. Later on, they extended this result to the case of two dimensional flow with nonzero vorticity in \cite{Bae2} through a two dimensional Helmholtz decomposition $\u=\nabla\phi+\nabla^{\bot}\psi$, and they also investigated the unique existence and stability of two-dimensional subsonic flows with self-gravitiation in \cite{Bae3}. In \cite{BaeWeng}, Bae and Weng found a Helmholtz decomposition for three dimensional axisymmetric vector field with the form
$$\u=\nabla\phi+\curl{\bf V}\quad\mbox{with} \quad{\bf V}=h{\bf e}_r+\psi{\bf e}_{\th},$$
where $\phi, h, \psi$ are functions of $(x, r)$ for $r=\sqrt{x_2^2+x_3^2}$. With this decomposition and careful decomposition of the boundary conditions, an axisymmetric subsonic flow with nonzero vorticity to the steady Euler-Poisson system was established. Different from this decomposition, Weng constructed a deformation-curl-Poisson decomposition for the three dimensional steady Euler-Poisson system in \cite{Weng4}, so as to handle the hyperbolic-elliptic coupled structure in subsonic region. The main issue of this decomposition is to represent the density as a function of the Bernoulli's function, the entropy, the electrostatic potential and the speed by using the the Bernoulli's law, and the continuity equation is rewritten as a restriction equation on the deformation matrix. This together with the curl system and the Poisson system form a deformation-curl-Poisson system which is elliptic in subsonic region in the sense of Agmon, Dougalis and Nirenberg \cite{Nirenberg1}. Based on this decomposition, the author in \cite{Weng4} further proved the structural stability of one dimensional subsonic flows under some three-dimensional perturbations on the entrance and exit of a rectangular cylinder.

With regard to the studies for structural stability of supersonic flows, Bae, Duan, Xie in \cite{Bae4} established the structural stability for supersonic solutions of the Euler-Poisson system for hydrodynamical model in semiconductor devices and plasmas in two-dimensional domain . Bae and Park in \cite{BaeYang2} proved the unique existence of supersonic solutions of the Euler-Poisson system for potential flow in a three-dimensional rectangular cylinder when prescribing the velocity and the strength of electric field at the entrance. In \cite{BaeYang1}, the authors constructed a family of radial transonic shock solutions for Euler-Poisson system in a two dimensional convergent nozzle, and studied various analytical features especially including the monotonicity property of the pressure at the exit with respect to shock location.
The Euler-Poisson system with relaxation effect is also surveyed. When the doping profile lies in the subsonic and supersonic regions respectively, the authors in \cite{LMZ1, LMZ2} considered different flow patterns for one dimensional steady Euler-Poisson system with relaxation effect supplemented with sonic boundary condition.

Wang and Xin proved the existence and uniqueness of the smooth transonic flow of Meyer type in de Laval nozzles in \cite{WX1, WX4, WX2, WX3}, and all sonic points are exceptional and characteristically degenerate. They studied the subsonic-sonic flows by using the Chaplygin equations in the plane of the velocity potential and the stream function as well as the comparison principles. Recently, Weng, Xin and Yuan studied cylindrical symmetric smooth transonic spiral flows by considering one side boundary value problem to the steady Euler system in a concentric cylinder in \cite{Weng1}. They found that, in the presence of the angular velocity, many new flow patterns will appear. Based on the cylindrical symmetric background transonic solutions established in \cite{Weng1}, they further investigated the structural stability of these special transonic flows in \cite{Weng2}, where the sonic points are all nonexceptional and noncharacteristic. They established the existence and uniqueness of smooth cylindrically symmetric transonic flows in a concentric cylinder under the cylindrical and axi-symmetric perturbations. To deal with the hyperbolicity in subsonic regions, they employed the deformation-curl decomposition for the steady Euler system introduced in \cite{Weng3} to prove the existence and uniqueness of smooth transonic rotational flows. More recently, Weng and Xin  in \cite{Weng7} established the existence and stability of cylindrical transonic shock solutions under three dimensional perturbations of the incoming flows and the exit pressure without any further restrictions on the background transonic shock solutions, where the deformation-curl decomposition to the steady Euler system plays a key role.

Let us formulate the problem to be considered. Introduce the cylindrical coordinates $(r, \th, x_3)$
$$
r=\sqrt{x_1^2+x_2^2}, \quad \th=\arctan\frac{x_2}{x_1}, \quad x_3=x_3,
$$
and the domain $\Omega$ concerned in this paper can be rewritten as
\be
\Omega=\{(r,\th,x_3):r_0<r<r_1,\th\in\mathbb{T}_{2\pi},-1<x_3<1\}.
\ee
The boundary $\partial\Omega$ is naturally divided into the inner cylinder part $\Gamma^{0}$, the outer cylinder part $\Gamma^{1}$, and the impermeable wall $\Gamma^w$, more precisely,
\begin{equation*}
	\begin{aligned}
		\Gamma^{0}&=\{(r,\th,x_3):r=r_0,\th\in\mathbb{T}_{2\pi},-1<x_3<1\},\\
		\Gamma^{1}&=\{(r,\th,x_3):r=r_1,\th\in\mathbb{T}_{2\pi},-1<x_3<1\},\\
		\Gamma^{w}&=
		\{(r,\th,x_3):r_0<r<r_1,\th\in\mathbb{T}_{2\pi},x_3=\pm 1\}.
	\end{aligned}
\end{equation*}
In the cylindrical coordinate, the velocity field can be represented as ${\bf u}=U_1 \mathbf{e}_r+U_2 \mathbf{e}_{\th}+U_3 \mathbf{e}_3$, where
\begin{equation*}
\mathbf{e}_r=\begin{pmatrix}
\cos\th\\ \sin\th \\0
\end{pmatrix},
\mathbf{e}_{\th}=\begin{pmatrix}
-\sin\th\\ \cos\th \\0
\end{pmatrix},
\mathbf{e}_{3}=\begin{pmatrix}
0\\ 0 \\1
\end{pmatrix}.
\end{equation*}
Then the steady compressible Euler-Poisson equations \eqref{EPS} in cylindrical coordinates read
\begin{eqnarray}\label{ProblemImm}
\begin{cases}
\partial_r(\rho U_1)+\frac{1}{r}\partial_{\theta}(\rho U_2)+\frac{1}{r} \rho U_1+\p_{x_3}(\rho U_3)=0,\\
\rho\left(U_1\partial_r +\frac{U_2}{r}\partial_{\theta}+U_3\p_{x_3}\right)U_1+\partial_r P-\frac{\rho U_2^2}{r}=\rho\p_r\Phi,\\
\rho\left(U_1\partial_r +\frac{U_2}{r}\partial_{\theta}+U_3\p_{x_3}\right)U_2+\frac{1}{r}\partial_{\theta}P+\frac{\rho U_1U_2}{r}=\rho\frac1 r\p_{\th}\Phi,\\
\rho\left(U_1\partial_r +\frac{U_2}{r}\partial_{\theta}+U_3\p_{x_3}\right)U_3+\p_{x_3}P=\rho\p_{x_3}\Phi,\\
\rho\left(U_1\partial_r +\frac{U_2}{r}\partial_{\theta}+U_3\p_{x_3}\right)A(S)=0,\\
\rho\left(U_1\partial_r +\frac{U_2}{r}\partial_{\theta}+U_3\p_{x_3}\right)K=0,\\
\left(\p_r^2+\frac 1 r\p_r+\frac{1}{r^2}\p_{\th}^2+\p_{x_3}^2\right)\Phi=\rho-b,
\end{cases}
\end{eqnarray}
where we define the generalized Bernoulli function $K(r,\theta,x_3)$ as
\be\no
K:=\frac{1}{2}|{\bf u}|^2+\frac{\gamma}{\gamma-1}\frac{P}{\rho}-\Phi.
\ee

Motivated by \cite{Weng1, Weng2}, we are interested in cylindrical symmetric smooth solutions to \eqref{ProblemImm}. However, due to the effect of the electric field, the solution behavior of \eqref{1d_EP} is much more complicated than the steady Euler case, a complete classification of all solutions with(without) shocks seems to be quite difficult. Here we first consider the following two simple cases.

Fix $b$ to be a constant $b_0>0$, we construct the background flow of the form $\u(x)=\bar U_1(r)\mathbf e_r+\bar U_2(r)\mathbf e_{\th}$, $\rho(x)=\bar\rho(r)$, and $A(x)=\bar A(r)$, $\Phi(x)=\bar \Phi(r)$ in $\Omega$, solving the following system
\begin{equation}\label{1d_EP}
	\begin{cases}
		\p_r(r \bar \rho  \bar U_1)=0,\\
     	 \bar U_1 \p_r\bar U_1+\frac{1}{\bar\rho} \p_r\bar P-\frac{1}{r} \bar U_2^2= \bar E,\\
    	 \bar U_1 \p_r\bar U_2+\frac{1}{r} \bar U_1\bar U_2=0,\\
     	 \bar U_1 \p_r\bar A=0,\\
    	 \p_r\bar E+\frac{1}{r}\bar E=\bar \rho-b_0,
	\end{cases}
\end{equation}
where $\bar E(r)=\bar\Phi'(r)$.

We prescribe the initial data at the entrance $r=r_0$ as
\be\label{bdc0}
\bar\rho(r_0)=\rho_0>0,\,
\bar U_{1}(r_0)=U_{10}>0,\,
\bar U_{2}(r_0)=U_{20},\, \bar A(r_0)=A_0>0,\,
\bar E(r_0)=E_0>0.
\ee
Define $m_1=r_0\rho_0U_{10}$ and $m_2=r_0U_{20}$, one then has
\be\no
\bar\rho\bar U_1=\frac{m_1}{r},\,\bar U_2=\frac{m_2}{r},\,
\bar A \equiv A_0,
\ee
The system \eqref{1d_EP} can be further reduced to the following equations for $(\bar\rho,\bar E)$
\begin{equation}\label{15}
	\begin{cases}
		\bar\rho'(r)=\bar\rho\cdot\frac{\bar U_1^2+\bar U_2^2+ \bar E r}{r(\bar c^2-\bar U_1^2)},\\
		(r\bar E)'(r)=r(\bar\rho-\bar b).
	\end{cases}
\end{equation}

We have the first proposition involving background subsonic flows.
\begin{proposition}\label{prop1}
Assume that $\rho_0>b_0$. Given constants $\rho_0>0, U_{10}>0, A_0>0, E_0>0$ satisfying $\ga A_0\rho_0^{\ga-1}>U_{10}^2+U_{20}^2$, then the initial value problem \eqref{1d_EP} and \eqref{bdc0} has a unique smooth subsonic solution $(\bar U_{1s}, \bar U_{2s}, \bar\rho_s, \bar A_{s}, \bar E_s)$ satisfying $|\bar\M(r)|^2<1$.
\end{proposition}
\begin{proof}

Denote the Mach numbers by
\be\no
\bar M_1^2(r)=\frac{\bar U_1^2(r)}{\bar c^2(\rho)},\,
\bar M_2^2(r)=\frac{\bar U_2^2(r)}{\bar c^2(\rho)},\,
\bar \M(r)=(\bar M_1(r), \bar M_2(r))^t.
\ee
A direct computation gives
\be\label{daoshu1}\begin{split}
&\frac{d}{d r}\bar M_1^2(r)=-\frac{\bar M_1^2}{r(1-\bar M_1^2)}
\left(
(\gamma-1)\bar M_1^2+(\gamma+1)\bar M_2^2+(\gamma+1)\frac{\bar E r}{\bar c^2}+2
\right),\\
&\frac{d}{d r}\bar M_2^2(r)=-\frac{\bar M_2^2}{r(1-\bar M_1^2)}
\left(
(\gamma-3)\bar M_1^2+(\gamma-1)\bar M_2^2+(\gamma-1)\frac{\bar E r}{\bar c^2}+2
\right),\no\\
&\frac{d}{d r}|\bar{\bf M}|^2(r)=-\frac{|\bar{\bf M}|^2}{r(1-\bar M_1^2)}
\left(
(\gamma-1)|\bar{\bf M}|^2+2
\right)-
\frac{ (\gamma+1)\bar M_1^2+(\gamma-1)\bar M_2^2}{\bar c^2(1-\bar M_1^2)}\bar E.\no
\end{split} \ee
Since
$\gamma A_0\rho_0^{\gamma-1}>U_{10}^2+U_{20}^2$ and
$\rho_0>b_0$, the existence and uniqueness of local solution to \eqref{15} is guaranteed by the standard theory. Then for some constant $\mu_0>0$, one has
$$\bar\rho'(r)>0,\,\quad\mbox{and}\quad \bar\rho(r)>b_0, \,\, \forall \,r\in[r_0, r_0+\mu_0],$$
which implies
$$(r\bar E)'=r(\bar\rho-b_0)>0,\,\mbox{and}\,\,r\bar E>0,\quad \forall\, r\in[r_0,r_0+\mu_0].$$
This tells that $\frac{d}{dr}(|\bar\M|^2)(r)<0, \forall r\in[r_0, r_0+\mu_0]$. Hence, $|\bar{\bf M}|^2(r_0+\mu_0)<1$, one can extend the solution of \eqref{15} to the whole interval $[r_0, r_1]$ and $0<|\bar{\bf M}|^2(r)<1$ for $\forall r\in[r_0, r_1]$. As a consequence, one obtains a unique smooth radially symmetric subsonic solution
$(\bar\rho_s,\bar U_{1s},\bar U_{2s}, \bar A_s,\bar E_s)$ with nonzero angular velocity on $[r_0,r_1]$.\end{proof}

Without loss generality, one may assume $\bar\Phi(r_0)=0$ and recover $\bar\Phi(r)$ by $\bar E(r)$ as
\be\label{18}
\bar\Phi(r)=\int_{r_0}^{r} \bar E(\tau) d\tau.
\ee
Set $\bar m:=\bar\rho\bar U_1$, and $\bar P(r)=A_0\bar\rho^{\ga}(r)$, we call $(\bar \rho_s, \bar U_{1s}, \bar U_{2s}, \bar A_{s}, \bar \Phi_s)$ the background subsonic solution. Next, the background transonic solution denoted by $(\bar \rho_t, \bar U_{1t}, \bar U_{2t}, \bar A_{t}, \bar \Phi_t)$ is established in the following proposition.
\begin{proposition}
 Assume that $\rho_0>b_0, A_0>0, E_0>0$. Then if the incoming flow is supersonic, but subsonic in the $r$-direction, i.e., $U_{10}^2<\ga A_0\rho_0^{\ga-1}<U_{10}^2+U_{20}^2$, there exists a unique smooth transonic flow $(\bar\rho_t(r), \bar U_{1t}(r), \bar U_{2t}(r), \bar A_{t}, \bar E_t(r))$ on $[r_0, r_1]$, if $r_1$ is large enough.
\end{proposition}
\begin{proof}
In view of \eqref{15}, one can deduce that for some constant $\mu_0>0$,
$$
\bar\rho'(r)>0, \quad\mbox{and}\quad \bar\rho(r)>b_0, \quad\forall r\in[r_0, r_0+\mu_0],
$$
which further indicated that
$$
(r\bar E(r))'>0, \quad\forall r\in[r_0, r_0+\mu_0],
$$
that is to say,
$$r\bar E(r)\geq r_0E_0>0,\quad\forall r\in[r_0, r_0+\mu_0].$$
Hence, \eqref{daoshu1} implies
\begin{align*}
&\frac{d}{dr}\bar M_1^2(r)<0, \quad \bar M_1^2(r)\leq\bar M_{10}^2=\frac{U_{10}^2}{c^2(\rho_0)}<1,\quad\forall r\in[r_0, r_0+\mu_0],\\
&\frac{d}{dr}|\bar\M|^2(r)\leq-\frac{2(\bar M_1^2(r)+\bar M_2^2(r))}{r(1-\bar M_1^2(r))}\leq-\frac{2(\bar M_1^2(r)+\bar M_2^2(r))}{r}, \quad\forall r\in[r_0, r_0+\mu_0].
\end{align*}

Then $|\bar\M|^2(r_0+\mu_0)<1$, one can extend the solution of \eqref{15} to $[r_0, +\infty)$, and the above inequalities still holds.

Denote $f(r)=\bar M_1^2(r)+\bar M_2^2(r)$, then we have $f(r)\leq\frac{r_0^2f(r_0)}{r^2}, \,\,\forall r\geq r_0$. Thus $(\bar M_1^2+\bar M_2^2)(r)\leq\frac{r_0^2(M_{10}^2+M_{20}^2)}{r^2}, \,\,\forall r\geq r_0$.

Therefore, there exists a unique $r_c\in(r_0, r_1)$ such that $(\bar M_1^2+\bar M_2^2)(r_c)=1$. And for $\forall r_1>r_c$, there exists a unique smooth transonic flow $(\bar\rho_t(r), \bar U_{1t}(r), \bar U_{2t}(r), \bar A_{t}, \bar E_t(r))$ on $[r_0, r_1]$.
\end{proof}

The purpose of this paper is not only to study the structural stability of smooth cylindrically symmetric subsonic flows with nonzero vorticity, but also the structural stability of axi-symmetric transonic flows for the steady Euler-Poisson system \eqref{EPS} in the concentric cylinder $\Omega$.

Overall, the paper is organized as follows. In Section \ref{sec1}, we recall a deformation-curl-Poisson decomposition for the steady Euler-Poisson system. In Section \ref{sec2}, we focus on investigating the structural stability of smooth cylindrical subsonic flows with nonzero vorticity under three dimensional perturbations on the inner and outer cylinder. The key ingredient is a special structure to steady Euler-Poisson equation which guarantees the existence and uniqueness of a second order elliptic system for the velocity potential and the electrostatic potential. The final section devotes to establish the structural stability of smooth transonic spiral flows under suitable axi-symmetric perturbations.

%\section{Smooth cylindrical subsonic flows with nonzero vorticity}

\section{The deformation-curl-Poisson decomposition in cylindrical coordinates}\label{sec1}

To avoid the difficulties resulted from the elliptic-hyperbolic mixed structure of the steady Euler-Poisson system in subsonic regions, we follow the idea developed in \cite{Weng3, Weng4} to construct a deformation-curl-Poisson decomposition to steady Euler-Poisson system in cylindrical coordinates.

On one hand, we treat the hyperbolic mode by identifying that the generalized Bernoulli's quantity and the entropy satisfy the following transport equations
\begin{equation}\label{transport}
	\begin{cases}
		(\p_r+\frac{U_2}{rU_1}\p_{\theta}+\frac{U_3}{U_1}\p_{x_3})K=0,\\
		(\p_r+\frac{U_2}{rU_1}\p_{\theta}+\frac{U_3}{U_1}\p_{x_3})A=0.
	\end{cases}
\end{equation}
In cylindrical coordinates, $\curl {\bf u}=\bm\omega(r,\theta,x_3)=\omega_1{\bf e}_r
+ \omega_2{\bf e}_{\theta}
+ \omega_3{\bf e}_{x_3}$,
where
\begin{equation}\label{xuandu}
	\omega_1=\frac{1}{r}\p_{\theta}U_3-\p_{x_3}U_2,\,\,
	\omega_2=\p_{x_3}U_1-\p_{x_r}U_3,\,\,
	\omega_3=\p_{r}U_2-\frac{1}{r}\p_{\theta}U_1+\frac{U_2}{r}.
\end{equation}
It follows from the equations of momentum conservation that
\begin{equation*}
	\begin{cases}
		U_1\omega_3-U_3\omega_1+\frac{1}{r}\p_{\theta}K
		-\frac{\rho^{\gamma-1}}{\gamma-1}\frac{1}{r}\p_{\theta}A=0,\\
		-U_1\omega_2+U_2\omega_1+\p_{x_3}K
		-\frac{\rho^{\gamma-1}}{\gamma-1}\p_{x_3}A=0.
	\end{cases}
\end{equation*}
Then $\omega_2, \omega_3$ are given by
\begin{equation}\label{xuandu2}
	\begin{cases}
		\omega_2=\frac{1}{U_1}(U_2\omega_1+\p_{x_3}K
		-\frac{\rho^{\gamma-1}}{\gamma-1}\p_{x_3}A),\\
		\omega_3=\frac{1}{U_1}(U_3\omega_1-\frac{1}{r}\p_{\theta}K
		+\frac{\rho^{\gamma-1}}{\gamma-1}\frac{1}{r}\p_{\theta}A).
	\end{cases}
\end{equation}
Plugging \eqref{xuandu2} into the following equation
$$\div \curl{\bf u}=\p_r\omega_1+\frac{1}{r}\p_{\theta}\omega_2+\p_{x_3}\omega_3
+\frac{\omega_1}{r}=0,$$
and performing direct calculations, one arrives at
\be\label{identify26} \begin{split}
	&\left(
	\p_r+\frac{U_2}{rU_1}\p_{\theta}+\frac{U_3}{U_1}\p_{x_3}
	\right)\omega_1
	+\left(
	\frac{1}{r}+\frac{1}{r}\p_{\theta}(\frac{U_2}{U_1})+\p_{x_3}(\frac{U_3}{U_1})
	\right)\omega_1\\
	&+\frac{1}{r}\p_{\theta}(\frac{1}{U_1})\p_{x_3}K
	-\p_{x_3}(\frac{1}{U_1})\frac{1}{r}\p_{\theta}K-\frac{1}{\gamma-1}\frac{1}{r}\p_{\theta}(\frac{\rho^{\gamma-1}}{U_1})\p_{x_3}A+\frac{1}{\gamma-1}\p_{x_3}(\frac{\rho^{\gamma-1}}{U_1})\frac{1}{r}\p_{\theta}A=0.
\end{split} \ee

On the other hand, the elliptic modes can be handled by invoking the Bernoulli' function to rewrite the density as the function of $A, K, \Phi, |\U|^2$:
\be\label{midu}
\rho=H(A, K, \Phi, |\U|^2)=\left(\frac{\ga-1}{\ga A}\right)^{\frac{1}{\ga-1}}\left(K+\Phi-\frac1 2|\U|^2\right)^{\frac{1}{\ga-1}}.
\ee
Substituting \eqref{midu} into the continuity equation yields
\be\label{zhu} \begin{split}
\p_r&\left(H(A, K, \Phi, |\U|^2)U_1\right)+\frac1 r\p_{\th}\left(H(A, K, \Phi, |\U|^2)U_2\right)\\
&+\p_{x_3}\left(H(A, K, \Phi, |\U|^2)U_3\right)+\frac1 rH(A, K, \Phi, |\U|^2)U_1=0,
\end{split} \ee
which can be rewritten as a Frobenious inner product of a symmetric matrix $\mathbb{M}$ with the deformation matrix
$$
\mathbb{M}({\bf K}, {\bf \Phi}, {\bf U}):\mathbb{D}({\bf U})+\frac{c^2(K, \Phi, {\bf U})}{r}{\bf U}_1=0,
$$
where $\mathbb{M}:\mathbb{D}=\sum\limits_{i, j=1}^3 m_{ij}d_{ij}$, $\forall \,\mathbb{M}=(m_{ij})_{i,j=1}^3, \mathbb{D}=(d_{ij})_{i,j=1}^3$, with
\begin{eqnarray}
	\mathbb{M}=
	\left(
	\begin{array}{ccc}
		c^2(K, \Phi, {\bf U})-U_1^2 & -U_1U_2 & -U_1U_3\\
		-U_1U_2 & c^2(K, \Phi, {\bf U})-U_2^2 & -U_2U_3\\
        -U_1U_3 & -U_2U_3 & c^2(K, \Phi, {\bf U})-U_3^2
	\end{array}
	\right),
\end{eqnarray}
\begin{eqnarray}
\mathbb{D}({\bf U})=
\left(
\begin{array}{ccc}
\p_rU_1 & \frac1 2(\p_rU_2+\frac1 r\p_{\th}U_1) & \frac1 2(\p_rU_3+\p_{x_3}U_1)\\
\frac1 2(\p_rU_2+\frac1 r\p_{\th}U_1) & \frac1 r \p_{\th}U_2 & \frac1 2(\frac1 r\p_{\th}U_3+\p_{x_3}U_2)\\
\frac1 2(\p_rU_3+\p_{x_3}U_1) & \frac1 2(\frac1 r\p_{\th}U_3+\p_{x_3}U_2) & \p_{x_3}U_3
\end{array}
\right).
\end{eqnarray}

Besides, the electrostatic potential satisfies
\begin{eqnarray}\label{potential}
	\Delta\Phi=H(A,K,\Phi,|U|^2)-b.
\end{eqnarray}

As a consequence, we obtain the following deformation-curl-Poisson system for the velocity field
\be\label{system}
\begin{cases}
\p_r\left(H(A, K, \Phi, |\U|^2)U_1\right)+\frac1 r\p_{\th}\left(H(A, K, \Phi, |\U|^2)U_2\right)\\
\quad\quad+\p_{x_3}\left(H(A, K, \Phi, |\U|^2)U_3\right)+\frac1 rH(A, K, \Phi, |\U|^2)U_1=0,\\
\frac{1}{r}\p_{\theta}U_3-\p_{x_3}U_2=\omega_1,\\
\p_{x_3}U_1-\p_{r}U_3=\omega_2,\\
\p_{r}U_2-\frac{1}{r}\p_{\theta}U_1+\frac{U_2}{r}=\omega_3,\\
\Delta\Phi=H(A,K,\Phi,|\U|^2)-b.
\end{cases}
\ee

Then one has the following Lemma.
\bl\label{equivalent}
{\bf (Equivalence)} Assume that $C^1$ smooth solutions $(\rho, \U, A, K, \Phi)$ to \eqref{EPS} defined on a domain $\Omega$ do not contain the vacuum and the radial velocity $U_1(r, x_3)>0$ in $\Omega$, then the steady Euler-Poisson equations \eqref{EPS} is equivalent to \eqref{transport}, \eqref{xuandu2}, \eqref{identify26}, \eqref{midu}, and \eqref{system}.
\el

\section{Structural stability of smooth cylindrical subsonic flows}\label{sec2}

This section is dedicated to establish the structural stability of the smooth cylindrical subsonic flows under the three-dimensional perturbations at the inner and outer cylinder by utilizing the deformation-curl-Poisson decomposition.

\subsection{Main Theorems}

We next examine the existence and uniqueness of smooth subsonic background solutions constructed in Proposition \ref{prop1} under two classes of boundary conditions.

Firstly, by setting
$$\bar m_{en}:=\bar\rho_s(r_0)\bar U_{1s}(r_0) \,\,\,\mbox{and}\,\,\, \bar m_{ex}:=\bar\rho_s(r_1)\bar U_{1s}(r_1),$$
we prescribe the boundary conditions on the inner cylinder $r=r_0$, consisting of the radical component of the momentum and the vorticity, the generalized Bernoulli' quantity, the entropy and the electric potential:
\begin{equation}\label{BC_r0}
	\begin{cases}
		(\rho U_1)(r_0,\theta,x_3)=\bar m_{en}+\epsilon m_{en}(\theta,x_3),\\
		\omega_1(r_0,\theta,x_3)=\epsilon \omega_{en}(\theta,x_3),\\
		K(r_0,\theta,x_3)=\bar K_{s}+\epsilon K_{en}(\theta,x_3),\\
		A(r_0,\theta,x_3)=\bar A_{s}+\epsilon A_{en}(\theta,x_3),\\
		\Phi(r_0,\theta,x_3)=\bar \Phi_{s}(r_0)+\epsilon \Phi_{en}(\theta,x_3),
	\end{cases}
\end{equation}
where $(m_{en}, K_{en}, A_{en}, \Phi_{en})\in C^{2, \alpha}(\Gamma^0)$, $\omega_{en}\in C^{1, \alpha}(\Gamma^0)$ with some $\alpha\in(0, 1)$. Furthermore, the compatibility conditions should be subjected as
\be\label{cp1}
\p_{x_3}(m_{en}, K_{en}, A_{en}, \Phi_{en})(\th, \pm1)=0,\quad \omega_{en}(\th,\pm 1)=0,\quad \forall \,\,\th\in[0,2\pi],
\ee
 And the conditions on the outer cylinder $r=r_1$ are imposed as
\begin{equation}\label{BC_r1}
	\begin{cases}
	(\rho U_1)(r_1,\theta,x_3)=\bar m_{ex}+\epsilon m_{ex}(\theta,x_3),\\
	\Phi(r_1,\theta,x_3)=\bar \Phi_s(r_1)+\epsilon \Phi_{ex}(\theta,x_3),
	\end{cases}
\end{equation}
here $(m_{ex},\,\Phi_{ex})\in (C^{2, \alpha}(\Gamma^1))^2$ and satisfies the compatibility conditions:
\be\label{cp2}
\p_{x_3}m_{ex}(\th,\pm 1)=\p_{x_3}\Phi_{ex}(\th,\pm 1)=0,
\ee
and
\be\label{cp3}
\int_{-1}^1\int_0^{2\pi} m_{ex}(\th,x_3)\,r_1
-m_{en}(\th,x_3)\,r_0\,d\th dx_3
=0.
\ee
In addition, the slip boundary conditions are given by
\begin{eqnarray}\label{BC_wall}
	U_3(r,\theta,\pm 1)=0,\quad \p_{x_3}\Phi(r,\theta,\pm 1)=0.
\end{eqnarray}
%Notation:
%\begin{eqnarray*}
%	\begin{split}
%		\Gamma_{en}:&=\{(r_0,\theta,x_3):0\leq\theta\leq2\pi,\,-1\leq x_3\leq 1\},\\
%		\Gamma_{ex}:&=\{(r_1,\theta,x_3):0\leq\theta\leq2\pi,\,-1\leq x_3\leq 1\},\\
%		\Gamma_{uw}:&=\{(r,\theta,x_3):r_0\leq r\leq r_1,\,0\leq\theta\leq2\pi\},\\
%		\Gamma_{lw}:&=\{(r,\theta,x_3):r_0\leq r\leq r_1,\,0\leq\theta\leq2\pi\}.
%	\end{split}
%\end{eqnarray*}

The ion density $b(x)$ is assumed to be a small perturbation of ion background density
\be
b(x)=b_0+\epsilon \tilde b(r,\th,x_3),
\ee
where $\tilde b(r,\th,x_3)\in C^\al(\Omega)$.

The first main result of our work is stated as follows.
\bt\label{theorem}
Given the boundary data
$(m_{en},K_{en},A_{en},\Phi_{en})\in C^{2,\al}(\Gamma^0)$,
$\omega_{en}\in C^{1,\al}(\Gamma^0)$,\\
$(m_{ex},\Phi_{ex})\in C^{2,\al}(\Gamma^1)$, and $\tilde b\in C^\al (\Omega)$ satisfying the compatibility conditions \eqref{cp1}, \eqref{cp2} and \eqref{cp3}, there exists a small constant $\epsilon_0>0$ that depends only on the background subsonic flow $(\bar\rho_s,\bar U_{1s},\bar U_{2s},\bar A_{s},\bar \Phi_s)$ and $(m_{en},\omega_{en},K_{en},A_{en},\Phi_{en},m_{ex},\Phi_{ex},\tilde b)$, such that for any $0<\epsilon<\epsilon_0$, there exists a smooth subsonic flow $(U_1,U_2,U_3,A,K,\Phi)\in (C^{2,\al}(\bar \Omega))^6$ to \eqref{ProblemImm} satisfying boundary conditions \eqref{BC_r0},\eqref{BC_r1} and \eqref{BC_wall}. Furthermore, the following estimate holds:
\begin{equation}
	\|(U_1,U_2,U_3,A,K,\Phi)-(\bar U_{1s},\bar U_{2s}, 0, \bar A_s, \bar K_s, \bar\Phi_s)\|_{C^{2,\al}(\bar \Omega)}\leq C_0\epsilon,
\end{equation}
where the positive constant $C_0$ only depends on the background subsonic flow, boundary data and $\al$.
\et

In fact, we can also prescribe the following second class of boundary conditions at the inner and outer cylinder as described in \cite{Weng4}. On the inner cylinder, we impose the generalized Bernoulli's quantity, the entropy, the electrostatic and the tangential velocity
\begin{equation}\label{BC_r00}
	\begin{cases}
		U_2(r_0,\theta,x_3)=\epsilon U_{2, en}(\theta,x_3),\\
		U_3(r_0,\theta,x_3)=\epsilon U_{3, en}(\theta,x_3),\\
		K(r_0,\theta,x_3)=\bar K_{s}+\epsilon K_{en}(\theta,x_3),\\
		A(r_0,\theta,x_3)=\bar A_{s}+\epsilon A_{en}(\theta,x_3),\\
		\Phi(r_0,\theta,x_3)=\bar\Phi_{s}(r_0)+\epsilon\Phi_{en}(\theta,x_3),
	\end{cases}
\end{equation}
where $(U_{2,en},U_{3,en},K_{en}, A_{en}, \Phi_{en})\in C^{2, \alpha}(\Gamma^0)$. And the following compatibility conditions also hold:
\be\label{cp111}
\begin{cases}
\p_{x_3}(U_{2, en},K_{en},A_{en},\Phi_{en})(\th, \pm1)=0,\quad \forall \,\,\th\in[0,2\pi],\\
U_{3, en}(\th, \pm 1)=\p_{x_3}^2U_{3, en}(\th, \pm1)=0.
\end{cases}
\ee
On the outer cylinder $r=r_1$, we subject the momentum and the electric potential as
\begin{equation}\label{BC_r11}
	\begin{cases}
	P(r_1,\theta,x_3)=\bar P(r_1)+\epsilon P_{ex}(\theta,x_3),\\
	\Phi(r_1,\theta,x_3)=\bar \Phi_{s}(r_1)+\epsilon \Phi_{ex}(\theta,x_3),
	\end{cases}
\end{equation}
here $(P_{ex}, \Phi_{ex})\in (C^{2, \alpha}(\Gamma^1))^2$ and satisfy the compatibility conditions:
\be\label{cp22}
\p_{x_3}P_{ex}(\th,\pm 1)=\p_{x_3}\Phi_{ex}(\th,\pm 1)=0.
\ee
Correspondingly, the existence and uniqueness result can be established as follows.
\bt\label{theoremm}
Given the boundary data
$(K_{en}, A_{en}, U_{2, en}, U_{3, en}, \Phi_{en})\in C^{2,\al}(\Gamma^0)$, $(P_{ex},\Phi_{ex})\in C^{2,\al}(\Gamma^1)$, and $\tilde b\in C^\al (\Omega)$ satisfying the compatibility conditions \eqref{cp111} and \eqref{cp22}, there exists a small constant $\epsilon_0>0$ which depends on the background subsonic flow $(\bar\rho_s, \bar U_{1s}, \bar U_{2s}, A_{0s}, \bar \Phi_s)$ and $(K_{en}, A_{en}, U_{2, en}, U_{3, en}, \Phi_{en}, P_{ex}, \Phi_{ex}, \tilde b)$, such that for any $0<\epsilon<\epsilon_0$, there exists a smooth subsonic flow $(U_1,U_2,U_3,A,K,\Phi)\in (C^{2,\al}(\bar \Omega))^6$ to \eqref{ProblemImm} satisfying boundary conditions \eqref{BC_r00},\eqref{BC_r11} and \eqref{BC_wall}. Furthermore, the following estimate holds:
\begin{equation}
	\|(U_1,U_2,U_3,A,K,\Phi)-(\bar U_{1s},\bar U_{2s}, 0, \bar A_s, \bar K_s, \bar \Phi_s)\|_{C^{2,\al}(\bar \Omega)}\leq C_0\epsilon,
\end{equation}
where the positive constant $C_0$ only depends on the background subsonic flow, boundary data and $\al$.
\et

\subsection{The proof of Theorem \ref{theorem}}\label{subsec1}

We begin by setting
\begin{eqnarray*}
	W_i=U_i-\bar U_{is},\,\,i=1,2,3,\quad W_4=A-\bar A_{s},\quad
	W_5=K-\bar K_{s},\quad W_6=\Phi-\bar\Phi_s,
\end{eqnarray*}
and defining the iteration space $\mathcal{S}_{\delta}$ as
\begin{equation}\label{itr-sp-sub}
	\begin{split}
			\mathcal{S}_{\delta}=\{
		{\bf W}=(W_1,W_2,W_3,W_4,W_5,W_6):\sum_{j=1}^{6}\|W_j\|_{C^{2,\al}(\bar\Omega)}\leq \delta\},
\end{split} \end{equation}
with the following compatibility conditions
\begin{equation}\label{subcomp}
		(W_3,\p_{x_3}^2W_3)(r,\th,\pm 1)=\p_{x_3}(W_1,W_2,W_4,W_5,W_6)(r,\th,\pm 1)=0,
\end{equation}
where $\delta$ is a suitably small positive constant to be determined later. It is worth noting that
\begin{equation*}
	\bar U_{3s}=0,\,\bar A_s=A_0,\,\bar K_s=K_0.
\end{equation*}
For ease of notation, we denote
$${\bf V}=(W_1,W_2,W_3),$$
which suggests
$${\bf U}={\bf V}+\bar{\bf U}_s.$$

We now propose to construct a new ${\bf W}\in\mathcal{S}_{\delta}$ by choosing any ${\bf W}^\sharp=(W_1^\sharp,\ldots, W_6^\sharp)\in\mathcal{S}_{\delta}$. In practice, we can proceed it through the following three steps.

{\bf Step 1.}\,
For the convenience of writing and reading, we first introduce the following notation
\begin{equation}
	\begin{aligned}
	R_1^{\sharp}:=R_1^{\sharp}(r,\theta,x_3)=
	\frac{1}{W_{1}^\sharp+\bar U_{1s}},\\
	R_2^{\sharp}:=R_2^{\sharp}(r,\theta,x_3)=\frac{W_{2}^\sharp+\bar U_{2s}}{W_{1}^\sharp+\bar U_{1s}},\\
	R_3^{\sharp}:=R_3^{\sharp}(r,\theta,x_3)=\frac{W_{3}^\sharp}{W_{1}^\sharp+\bar U_{1s}}.
	\end{aligned}
\end{equation}
Then, the equations of entropy function $A$ and generalized Bernoulli function $K$ are changed into
\begin{eqnarray}\label{eq_W4}
	\begin{cases}
(\p_r+R_2^\sharp\frac{1}{r}\p_{\theta}+R_3^\sharp\p_{x_3})
W_4=0,\\
W_4(r_0,\theta,x_3)=\epsilon A_{en}(\theta,x_3),
	\end{cases}
\end{eqnarray}
and
\begin{eqnarray}\label{eq_W5}
	\begin{cases}
(\p_r+R_2^\sharp\frac{1}{r}\p_{\theta}+R_3^\sharp\p_{x_3})
W_5=0,\\
		W_5(r_0,\theta,x_3)=\epsilon K_{en}(\theta,x_3).
	\end{cases}
\end{eqnarray}

For any point $(r,\theta,x_3)\in \bar\Omega$, we consider the following ODE system
%for $(z_1,z_2(z_1;r,\theta,x_3),z_3(z_1;r,\theta,x_3))$
\begin{equation*}
\begin{cases}
		\frac{d}{dz_1}z_2(z_1;r,\theta,x_3)=
		\frac{1}{z_1}R_2^\sharp(z_1,z_2(z_1;r,\theta,x_3),z_3(z_1;r,\theta,x_3)),
		\\
		\frac{d}{d z_1}z_3(z_1;r,\theta,x_3)=
		R_3^\sharp(z_1,z_2(z_1;r,\theta,x_3),z_3(z_1;r,\theta,x_3)),\\
		z_2(r;r,\theta,x_3)=\theta,\\
        z_3(r;r,\theta,x_3)=x_3.
\end{cases}
\end{equation*}
Denote $z_{i}^0(r,\theta,x_3)=z_i(r_0;r,\theta,x_3),i=2,3$, one deduces that
$(z_2(r_0;r,\theta,x_3),z_3(r_0;r,\theta,x_3))\in C^{2,\al}(\bar \Omega)$, owing to $W^\sharp_i\in C^{2,\al}(\bar \Omega),\,i=1,2,3$, which means that there exists a positive constant $C$ such that
\begin{eqnarray*}
	\|(z_2^0(r,\theta,x_3),z_3^0(r,\theta,x_3))\|_{C^{2,\al}(\bar\Omega)}\leq C,
\end{eqnarray*}
and thereby
\begin{eqnarray*}
	\begin{cases}
		W_4(r,\theta,x_3)=\epsilon A_{en}(z_2^0,z_3^0),\\
	W_5(r,\theta,x_3)=\epsilon K_{en}(z_2^0,z_3^0),
	\end{cases}
\end{eqnarray*}
with the estimate holding
\begin{equation}\label{est_W4_W5}
	\|(W_4,W_5)\|_{C^{2,\al}(\bar \Omega)}\leq C\epsilon\|(A_{en},K_{en})\|_{C^{2,\al}(\bar \Omega)}\leq C\epsilon,
\end{equation}
and the compatibility conditions
\begin{equation}\label{cp_W4W5}
	\p_{x_3}(W_4,W_5)(r,\th,\pm 1)=0.
\end{equation}
Indeed, thanks to the compatibility conditions
\begin{equation*}
	\p_{x_3}W_{i}^\sharp(r,\th,\pm 1)=W_3^\sharp(r,\th,\pm 1)=0,\,i=1,2,
\end{equation*}
one can easily derive
\begin{eqnarray*}
	\frac{\p z_2^0}{\p x_3}(r,\th,\pm 1)=0,\quad z_3^0(r,\th,\pm 1)=\pm 1.
\end{eqnarray*}
Consequently, a direct computation shows that
$$\p_{x_3}W_4(r,\th,\pm1)=
\p_{z_2^0}A_{en}(z_2^0,z_3^0)\frac{\p z_2^0}{\p x_3}(r,\th,\pm 1)
+\p_{z_3^0}A_{en}(z_2^0,z_3^0)\frac{\p z_3^0}{\p x_3}(r,\th,\pm 1)=0,$$
by using $\p_{x_3}A_{en}(\th,\pm 1)=0.$ Obviously, the same conclusion also holds true for $W_5$.

{\bf Step 2.} We next solve the following transport equation for the first component $\omega_1$ of the vorticity:
\begin{eqnarray}\label{eq_omg1}
	\begin{cases}
			(
		\p_r
		+R_2^\sharp\frac{1}{r}\p_{\theta}
		+R_3^\sharp\p_{x_3}
		)\omega_1+
		a_0^\sharp(r,\theta,x_3)\omega_1=
		F_0^\sharp(r,\theta,x_3),\\
		\omega_1(r_0,\theta,x_3)=\epsilon \omega_{en}(\theta,x_3),
	\end{cases}
\end{eqnarray}
where
\begin{align*}
&a_0^\sharp(r,\theta,x_3)=\frac{1}{r}
+\frac{1}{r}\p_{\theta}R_2^\sharp+\p_{x_3}R_3^\sharp,\\
&F_0^\sharp(r,\theta,x_3)=
\p_{x_3}R_1^\sharp \frac{1}{r}\p_{\theta}W_5
-\frac{1}{r}\p_{\theta}R_1^\sharp\p_{x_3}W_5\\
&\quad\quad\quad\quad\quad\quad\quad\quad
+\frac{1}{\gamma-1}\frac{1}{r}\p_{\theta}(R_1^\sharp H_\sharp^{\gamma-1})\p_{x_3}W_4
-\frac{1}{\gamma-1}\p_{x_3}(R_1^\sharp H_\sharp^{\gamma-1})\frac{1}{r}\p_{\theta}W_4,\\
&H_\sharp(r,\theta,x_3)=H(r,\theta,x_3;W_1^\sharp,W_2^\sharp,W_3^\sharp,W_4,W_5,W_6^\sharp),
\end{align*}
with
\begin{eqnarray*}
	H(r,\theta,x_3;{\bf W})=\left(
	\frac{\gamma-1}{\ga(W_4+A_0)}\left(W_5+K_0+W_6+\bar\Phi_s-\frac{1}{2}\left((W_1+\bar U_{1s})^2+(W_2+\bar U_{2s})^2+W_3^2\right)\right)
	\right)^\frac{1}{\gamma-1}.
\end{eqnarray*}

The characteristic method gives
\begin{align*}
		\omega_1(r,\theta,x_3)=&
	\epsilon\omega_{en}(z_2^0,z_3^0)
	e^{-\int_{r_0}^{r} a_0^\sharp(\tau,z_2(\tau;r,\theta,x_3),z_3(\tau;r,\theta,x_3))\,d\tau}\\
	&+\int_{r_0}^{r}F_{0}^\sharp(s,z_2(s;r,\theta,x_3),
	z_3(s;r,\theta,x_3)
	)	e^{-\int_{s}^{r} a_0^\sharp(\tau,z_2(\tau;r,\theta,x_3),z_3(\tau;r,\theta,x_3))\,d\tau} ds,
\end{align*}
from which one asserts
\begin{eqnarray}\label{est_omg1}
	\|\omega_1\|_{C^{1,\al}(\bar{\Omega})}\leq C(\epsilon+
	\delta^2+\delta\epsilon).
\end{eqnarray}

Since  ${\bf W}^\sharp\in\mathcal{S}_{\delta}$ satisfies the compatibility conditions, one can check that
$$F_0^\sharp(r,\th,\pm 1)=0,$$
by using \eqref{cp_W4W5}. Combining with $\omega_{en}(\th,\pm 1)=0$ and $z_3^0(r,\th,\pm 1)=\pm 1$, we have the following compatibility conditions for $\omega_1$
\begin{equation}\label{cp_omega1}
	\omega_1(r,\th,\pm 1)=0.
\end{equation}

Inserting \eqref{est_omg1} into \eqref{xuandu2} leads to
\begin{eqnarray}\label{eq_omg23}
\begin{cases}
\omega_2=R_2^\sharp \omega_1+R_1^\sharp\p_{x_3}W_5
-\frac{R_1^\sharp H_{\sharp}^{\gamma-1}}{\gamma-1}\p_{x_3}W_4,\\
\omega_3=R_3^\sharp\omega_1-R_1^\sharp \frac{1}{r}\p_{\theta}W_5
+\frac{R_1^\sharp H_{\sharp}^{\gamma-1}}{\gamma-1}\frac{1}{r}\p_{\theta}W_4.
\end{cases}
\end{eqnarray}
A straightforward computation yields that
\begin{equation}
	\sum_{j=2}^{3}\|\omega_j\|_{C^{1,\al}(\bar \Omega)}\leq C(\epsilon+\delta^2+\delta\epsilon).
\end{equation}

Using \eqref{cp_W4W5} and \eqref{cp_omega1}, one can also verify the following compatibility conditions:
\begin{eqnarray}\label{cp_omega23}
	\omega_2(r,\th,\pm 1)=\p_{x_3}\omega_3(r,\th,\pm 1)=0.
\end{eqnarray}
Moreover, one can easily check that
\begin{equation}\label{weiling}
	\begin{aligned}
		\div_{(r,\theta,x_3)}{\bm{\omega}}&=\frac{1}{r}\p_{r}(r\omega_1)+\frac{1}{r}\p_{\theta}\omega_2+\p_{x_3}\omega_3\\
&=\p_{r}\omega_1+\frac{1}{r}\omega_1+\frac{1}{r}\p_{\theta}
		\left(
R_2^\sharp \omega_1+R_1^\sharp\p_{x_3}W_5
-\frac{R_1^\sharp H_{\sharp}^{\gamma-1}}{\gamma-1}\p_{x_3}W_4
		\right)\\
		&\quad+\p_{x_3}
		\left(
R_3^\sharp\omega_1-R_1^\sharp \frac{1}{r}\p_{\theta}W_5
+\frac{R_1^\sharp H_{\sharp}^{\gamma-1}}{\gamma-1}\frac{1}{r}\p_{\theta}W_4
		\right)\\
		&=
		(\p_{r}+R_2^\sharp\frac{1}{r}\p_{\theta}+R_3^\sharp\p_{x_3})\omega_1+
		(\frac{1}{r}+\frac{1}{r}\p_{\theta}R_2^\sharp+\p_{x_3}R_3^\sharp)\omega_1\\
		&\quad+\frac{1}{r}\p_{\theta}R_1^\sharp\p_{x_3}W_5
	-\frac{1}{\gamma-1}\frac{1}{r}\p_{\theta}(R_1^\sharp H_\sharp^{\gamma-1})\p_{x_3}W_4\\
	&\quad-\p_{x_3}R_1^\sharp \frac{1}{r}\p_{\theta}W_5
	+\frac{1}{\gamma-1}\p_{x_3}(R_1^\sharp H_\sharp^{\gamma-1})\frac{1}{r}\p_{\theta}W_4=0.
	\end{aligned}
\end{equation}

{\bf Step 3.} We now intend to solve the velocity field and electrostatic potential.
By virtue of \eqref{midu}, one computes that
\begin{equation}\label{mimi} \begin{split}
\rho&=H(W_4+A_0, W_5+K_0, W_6+\bar\Phi_s,|{\bf V}+\bar {\bf U}_s|^2)\\
%&=\left(
%\frac{\gamma-1}{(W_4+\bar A)\gamma}(W_5+\bar K+W_6+\bar \Phi-\frac{1}{2}((\bar U_1+W_1)^2+(\bar U_2+W_2)^2+W_3^2))
%\right)^\frac{1}{\gamma-1}\\
&=\bar\rho_s
-\frac{\bar\rho_s \bar U_{1s}}{\bar c_s^2} W_1-\frac{\bar\rho_s \bar U_{2s}}{\bar c_s^2} W_2-\frac{\bar\rho_s}{(\gamma-1)A_0}W_4+\frac{\bar\rho_s}{\bar c_s^2}W_5+\frac{\bar\rho_s}{\bar c_s^2}W_6+O(|{\bf W}|^2),
\end{split}  \end{equation}
and
\begin{equation}\label{shengsu}  \begin{split}
	&\quad c^2(W_5+K_0,W_6+\bar\Phi_s,|{\bf V}+\bar {\bf U}_s|^2)\\
%&=\gamma A\rho^{\gamma-1}=(\gamma-1)\left(W_5+\bar K-\frac{1}{2}((W_1+\bar U_1)^2+(W_2+\bar U_2)^2+W_3^2)+W_6+\bar\Phi\right)\\
%&=(\gamma-1)\left(W_5+\bar K-\frac{1}{2}(W_1^2+2 \bar U_1 W_1+\bar U_1^2+W_2^2+2 \bar U_2 W_2+\bar U_2^2+W_3^2)+W_6+\bar\Phi\right)\\
&=\bar c_s^2+(\gamma-1)\left(W_5+W_6-\bar U_{1s} W_1-\bar U_{2s} W_2\right)-\frac{1}{2}(\gamma-1)(W_1^2+W_2^2+W_3^2),
\end{split}  \end{equation}
where
\begin{eqnarray*}\no
\bar\rho_s=H(A_0,K_0,\bar\Phi_s,|\bar \U_s|^2), \quad \bar c_s=c(\bar\rho_s, K_0, \bar\Phi_s).
\end{eqnarray*}
Let us for simplicity denote $Q(W_4,W_5)(r,\theta,x_3)=:\frac{\bar\rho_s}{(\gamma-1)A_0}W_4-\frac{\bar\rho_s}{\bar c_s^2}W_5$, then one derives by substituting \eqref{mimi} and \eqref{shengsu} into \eqref{zhu} that
\begin{eqnarray}\label{def_curl}
	\begin{cases}
		\p_{r}(r\bar\rho_s(1-\bar M_{1s}^2)W_1)
		-\p_{r}\left(r\bar\rho_s\bar M_{1s}\bar M_{2s} W_2\right)
		+\p_{\theta}(\bar\rho_s(1-\bar M_{2s}^2)W_2)\\
		\quad-\p_{\theta}\left(\bar\rho_s\bar M_{1s}\bar M_{2s}W_1\right)
		+\p_{x_3}(r\bar\rho_s W_3)
		+\p_{r}\left(\frac{r\bar\rho_s \bar U_{1s}}{\bar c_s^2}W_6\right)
		+\p_{\theta}\left(\frac{\bar\rho_s \bar U_{2s}}{\bar c_s^2}W_6\right)\\
		=
		\p_{r}(r\bar U_{1s} Q(W_4,W_5))			
		+\p_{\theta}(\bar U_{2s} Q(W_4,W_5))+
		\p_{r}(rF_1^\sharp)+\p_{\theta}F_2^\sharp+\p_{x_3}(rF_3^\sharp),\\
		\frac{1}{r}\p_{\theta} W_3-\p_{x_3} W_2=\omega_1,\\
		\p_{x_3} W_1-\p_{r} W_3=\omega_2,\\
	\frac{1}{r}\p_{r}(rW_2)-\frac{1}{r}\p_{\theta} W_1=\omega_3,\\
		(\p_{r}^2+\frac{1}{r}\p_r+\frac{1}{r^2}\p_{\th}+\p_{x_3}^2)W_6+\frac{\bar\rho_s \bar U_{1s}}{\bar c_s^2} W_1+ \frac{\bar\rho_s \bar U_{2s}}{\bar c_s^2} W_2-\frac{\bar\rho_s}{\bar c_s^2}W_6=-Q(W_4,W_5)+G^\sharp,
	\end{cases}
\end{eqnarray}
equipped with the following boundary conditions
\begin{eqnarray}\label{bc_def_curl}
\begin{cases}
	\left(\bar\rho_s(1-\bar M_{1s}^2) W_1-\bar\rho_s\bar M_{1s}\bar M_{2s}W_2\right)(r_0, \th, x_3)=g_0^\sharp(\theta,x_3),\\
	\left(\bar\rho_s(1-\bar M_{1s}^2) W_1-\bar\rho_s\bar M_{1s}\bar M_{2s}W_2\right)(r_1, \th, x_3)=g_1^\sharp(\theta,x_3)-\ka,\\
	W_3(r,\theta,\pm 1)=0,\\
	W_6(r_0,\theta,x_3)=\epsilon \Phi_{en}(\theta,x_3),\\
	W_6(r_1,\theta,x_3)=\epsilon \Phi_{ex}(\theta,x_3),\\
	\p_{x_3}W_6(r,\theta,\pm 1)=0,
\end{cases}
\end{eqnarray}
associated with
\begin{eqnarray*}
	\begin{split}
			g_0^\sharp(\theta,x_3;{\bf V}^\sharp)&=\epsilon \left(
		m_{en}+\frac{\bar\rho_s\bar U_{1s}}{(\gamma-1)A_0}A_{en}-\frac{\bar\rho_s\bar U_{1s}}{\bar c_s^2}K_{en}-\Phi_{en}\right)\\
		&\quad-\bigg(
		H(\epsilon A_{en}+A_0,\epsilon k_{en}+K_0,\epsilon \Phi_{en}+\bar\Phi_s,|{\bf V}^\sharp+\bar\U_s|^2)
		-H(A_0,K_0,\bar\Phi_s,|\bar\U_s|^2)
		\bigg)
		W_1^\sharp\bigg|_{r=r_0}\\
		&\quad-\bigg(
		H(\epsilon A_{en}+A_0,\epsilon k_{en}+K_0,\epsilon\Phi_{en}+\bar\Phi_s,|{\bf V}^\sharp+\bar\U_s|^2)
		-H(A_0,K_0,\bar\Phi_s,|\bar\U_s|^2)\\
		&\quad+\frac{\bar\rho_s\bar U_{1s}}{\bar c_s^2}W_1^\sharp
		+\frac{\bar\rho_s\bar U_{2s}}{\bar c_s^2}W_2^\sharp
		-\frac{\bar\rho_s}{\bar c_s^2} \epsilon\Phi_{en}
		+\frac{\bar\rho_s}{(\gamma-1)A_0} \epsilon A_{en}
		-\frac{\bar\rho_s}{\bar c_s^2}\epsilon K_{en}
		\bigg)
		\bar U_{1s}\bigg|_{r=r_0},
\end{split}
\end{eqnarray*}
\begin{eqnarray*}
	\begin{split}
			g_1^\sharp(\theta,x_3;{\bf V}^\sharp)&=\epsilon \left(
m_{ex}-\frac{\bar\rho_s \bar U_{1s}}{\bar c_s^2}\Phi_{ex}\right)
+\bar U_{1s} Q(W_4,W_5)\bigg|_{r=r_1}\\
&\quad-\bigg(
H(W_4+A_0,W_5+K_0,\epsilon\Phi_{ex}+\bar\Phi_s,|{\bf V}^\sharp+\bar \U_s|^2)
-H(A_0,K_0,\bar\Phi_s,|\bar \U_s|^2)
\bigg)
W_1^\sharp\bigg|_{r=r_1}\\
&\quad-\bigg(
H(W_4+A_0,W_5+K_0,\epsilon\Phi_{ex}+\bar\Phi_s,|{\bf V}^\sharp+\bar \U_s|^2)
-H(A_0,K_0,\bar\Phi_s,|\bar\U_s|^2)\\
&\quad+\frac{\bar\rho_s\bar U_{1s}}{\bar c_s^2}W_1^\sharp
+\frac{\bar\rho_s\bar U_{2s}}{\bar c_s^2}W_2^\sharp
-\frac{\bar\rho_s}{\bar c_s^2}\epsilon\Phi_{ex}
+\frac{\bar\rho_s}{(\gamma-1)A_0} W_4
-\frac{\bar\rho_s}{\bar c_s^2}W_5
\bigg)
\bar U_{1s}\bigg|_{r=r_1},	
	\end{split}
\end{eqnarray*}
where
\begin{eqnarray*}
\begin{split}
	F_i^\sharp(r,\theta,x_3;{\bf V}^\sharp,W_6^\sharp)&=-\left(
	H(W_4+A_0,W_5+K_0,W_6^\sharp+\bar\Phi_s,|{\bf V}^\sharp+\bar \U_s|^2)
	-H(A_0,K_0,\bar\Phi_s,|\bar \U_s|^2)
	\right)	W_i^\sharp\\
	&\quad-\bigg(
	H(W_4+A_0,W_5+K_0,W_6^\sharp+\bar\Phi_s,|{\bf V}^\sharp+\bar \U_s|^2)
	-H(A_0,K_0,\bar\Phi_s,|\bar \U_s|^2)\\
	&\quad+\frac{\bar\rho_s\bar U_{1s}}{\bar c_s^2}W_1^\sharp
	+\frac{\bar\rho_s\bar U_{2s}}{\bar c_s^2}W_2^\sharp
	-\frac{\bar\rho_s}{\bar c_s^2}W_6^\sharp
	+\frac{\bar\rho_s}{(\gamma-1)\bar A} W_4
	-\frac{\bar\rho_s}{\bar c_s^2}W_5
	\bigg)
	\bar U_i,\,i=1,2,\\
	F_3^\sharp(r,\theta,x_3;{\bf V}^\sharp,W_6^\sharp)&=-\left(
	H(W_4+A_0,W_5+\bar K,W_6^\sharp+\bar\Phi_s,|{\bf V}^\sharp+\bar \U_s|^2)
	-H(A_0,K_0,\bar\Phi_s,|\bar \U_s|^2)
	\right)	W_3^\sharp,\\
	G^\sharp(r,\theta,x_3;{\bf V}^\sharp,W_6^\sharp) &=H(W_4+A_0,W_5+K_0,W_6^\sharp+\bar\Phi_s,|{\bf V}^\sharp+\bar \U_s|^2)
	-H(A_0,K_0,\bar\Phi_s,|\bar \U_s|^2)\\
	&\quad+\frac{\bar\rho_s\bar U_{1s}}{\bar c_s^2}W_1^\sharp
	+\frac{\bar\rho_s\bar U_{2s}}{\bar c_s^2}W_2^\sharp
	-\frac{\bar\rho_s}{\bar c_s^2}W_6^\sharp
	+\frac{\bar\rho_s}{(\gamma-1)A_0} W_4
	-\frac{\bar\rho_s}{\bar c_s^2}W_5-\epsilon \tilde b(r,\theta,x_3),
\end{split}
\end{eqnarray*}
where we have used $b=\bar b(r)+\epsilon \tilde b$,
and the constant $\ka$ is represented by
\begin{equation}
	\begin{split}
		4\pi r_1\ka&=
		\int_{\{r=r_1\}} r_1\left[g_1^\sharp+\frac{\bar\rho_s\bar U_{1s}}{\bar c_s^2}\Phi_{ex}-\bar U_{1s}Q(W_4,W_5)-F_1^\sharp\right]_{r=r_1}\,
		d\theta dx_3\\
	&\quad-\int_{\{r=r_0\}} r_0\left[g_0^\sharp+\frac{\bar\rho_s\bar U_{1s}}{\bar c_s^2}\Phi_{en}-\bar U_{1s}Q(W_4,W_5)-F_1^\sharp\right]_{r=r_0}\,
	d\theta dx_3.
	\end{split}
\end{equation}

By \eqref{weiling}, one has $\div_{(r,\theta,x_3)}{\bm{\omega}}=0$. Consider the homogeneous system %$\mathfrak{W}=(\mathfrak{W}_1,\mathfrak{W}_2,\mathfrak{W}_3)$
\begin{eqnarray}\label{frak_W}
	\begin{cases}
		\p_r B_1+\frac1 rB_1+\frac1 r\p_{\th}B_2+\p_{x_3}B_3=0,\,\,\text{in}\,\Omega\\
		\frac{1}{r}\p_{\theta}B_3
		-\p_{x_3}B_2=0,\,\,\text{in}\,\Omega\\
		\p_{x_3}B_1
		-\p_{r}B_3=0,\,\,\text{in}\,\Omega\\
		\p_{r}B_2
		-\frac{1}{r}\p_{\theta}B_1+\frac{B_2}{r}=0,\,\,\text{in}\,\Omega\\
		B_1(r_0, \th, x_3)=B_1(r_1, \th, x_3)=0,\quad B_3(r, \th, \pm 1)=0.
    \end{cases}
\end{eqnarray}

Although $\Omega$ is non-simply-connected, we can conclude that, by Proposition 3.14 in \cite{abdg98}, the linear space
\begin{equation}
	K_T(\Omega)=\{
	B\in (L^2(\Omega))^3:B\, \text{satisfies}\,\eqref{frak_W}
	\}
\end{equation}
is not empty and has dimension 1. Then, by using the results in \cite{cs17, ky09}, there exists a unique smooth vector field $(\tilde W_1,\tilde W_2,\tilde W_3)\in (C^{2,\al}(\Omega))^3$ to
\begin{eqnarray}\label{div_curl-tilde}
	\begin{cases}
		\div_{(r,\theta,x_3)} \tilde{\bf W}=0,\\
		\frac{1}{r}\p_{\theta}\tilde W_3-\p_{x_3}\tilde W_2=\omega_1,\\
		\p_{x_3}\tilde W_1-\p_{r}\tilde W_3=\omega_2,\\
		\p_{r}\tilde W_2-\frac{1}{r}\p_{\theta}\tilde W_1+\frac{\tilde W_2}{r}=\omega_3,\\
		\tilde W_1(r_0,\theta,x_3)=0,\,\,\tilde W_1(r_1,\theta,x_3)=0,\,\,\tilde W_3(r,\th,\pm 1)=0,\\
		\tilde W_i(r,\th,x_3)=\tilde W_i(r,\th+2\pi,x_3),i=1,2,3,
	\end{cases}
\end{eqnarray}
which satisfies the following properties:
\begin{itemize}
  \item $\tilde W_1 \vec{e}_r+\tilde W_2\vec{e}_{\th}+\tilde W_3\vec{e}_3$ is perpendicular to $K_T(\Omega)$ in the sense of the inner product of $(L^2(\Omega))^3.$
  \item $(\tilde W_1,\tilde W_2,\tilde W_3)$ satisfies the estimates
\begin{equation}
	\sum_{i=1}^{3}\|\tilde W_i\|_{C^{2,\al}(\bar \Omega)}\leq
	C \sum_{i=1}^{3}\|\omega_i\|_{C^{1,\al}(\bar \Omega)},
\end{equation}
\end{itemize}
and the compatibility conditions:
\begin{equation}\label{cp_tilde123}
	\p_{x_3}(\tilde W_1,\tilde W_2)(r,\th,\pm 1)
	=(\tilde W_3,\p_{x_3}^2\tilde W_3)(r,\th,\pm 1)=0,
\end{equation}
thanks to \eqref{cp_omega1} and \eqref{cp_omega23}.

Setting $\check{W}_j=W_j-\tilde W_j$, $j=1, 2, 3$, thus there holds $\curl\check{W}=0$. One may assume the existence of a potential function $\psi$ such that
\begin{equation*}
\check{W}=\nabla_{(r,\theta,x_3)}\psi=(\p_{r}\psi,\frac{1}{r}\p_{\theta}\psi,\p_{x_3}\psi).
\end{equation*}

We next proceed by replacing
$$
W_1=\p_{r}\psi+\tilde W_1,\quad
	W_2=\frac{1}{r}\p_{\theta}\psi+\tilde W_2,\quad
	W_3=\p_{x_3}\psi+\tilde W_3,
$$
into \eqref{def_curl}, and performing some direct calculations, to derive by taking $r=z_1,\theta=z_2,x_3=z_3$ that
\begin{eqnarray}\label{psi_W6}
	\begin{cases}
	\p_{z_i}(\mathcal{A}_{ij}\p_{z_j}\psi)
+\p_{z_1}\left(\frac{z_1\bar\rho_s\bar U_{1s}}{\bar c_s^2} W_6\right)
+\p_{z_2}\left(\frac{\bar\rho_s\bar U_{2s}}{\bar c_s^2} W_6\right)
=\div_z{\tilde {\bf F}},\\
		\p_{z_i}(\mathcal{B}_{ij}\p_{z_j}W_6)+\frac{z_1\bar\rho_s \bar U_{1s}}{\bar c_s^2} \p_{z_1}\psi+ \frac{\bar\rho_s \bar U_{2s}}{\bar c_s^2} \p_{z_2}\psi-\frac{z_1\bar\rho_s}{\bar c_s^2}W_6=\tilde G,
	\end{cases}
\end{eqnarray}
with the corresponding boundary conditions
\begin{eqnarray}\label{BC_psi_W6}
\begin{cases}
	\left(r_0\bar\rho_s(1-\bar M_{1s}^2)\p_{z_1}\psi-\bar\rho_s \bar M_{1s}\bar M_{2s}\p_{z_2}\psi\right)(r_0,z')=\tilde g_0(z'),\\
	\left(r_1\bar\rho_s(1-\bar M_{2s}^2)\p_{z_1}\psi-\bar\rho_s \bar M_{1s}\bar M_{2s}\p_{z_2}\psi\right)(r_1,z')=\tilde g_1(z'),\\
	\p_{z_3}\psi(z_1,z_2,\pm 1)=0,\\
	W_6(r_0,z')=\epsilon\Phi_{en}(z'),\\
	W_6(r_1,z')=\epsilon \Phi_{ex}(z'),\\
	\p_{z_3}W_6(z_1,z_2,\pm 1)=0,
\end{cases}
\end{eqnarray}
where $z'=(z_2,z_3)$ and
\begin{eqnarray*}
	\begin{split}
	&\tilde F_1=r
	\left(
	F_1^\sharp-\bar\rho_s(1-\bar M_{1s}^2)\tilde W_1
		+\bar\rho_s\bar M_{1s}\bar M_{2s} \tilde W_2+\frac{\bar\rho_s\bar U_{1s}}{(\gamma-1)A_0}W_4-\frac{\bar\rho_s\bar U_{1s}}{\bar c_s^2}W_5
		\right),\\
		&\tilde F_2=F_2^\sharp-\bar\rho_s(1-\bar M_{2s}^2)\tilde W_2
		+\bar\rho_s\bar M_{1s}\bar M_{2s} \tilde W_1+\frac{\bar\rho_s\bar U_{2s}}{(\gamma-1)A_0}W_4-\frac{\bar\rho_s\bar U_{2s}}{\bar c_s^2}W_5,\\
		&\tilde F_3=r(F_3^\sharp-\bar\rho_s\tilde W_3),\\
		&\tilde G=r\left(
		G^\sharp-Q(W_4,W_5)
		-\frac{\bar\rho_s\bar U_{1s}}{\bar c_s^2}\tilde W_1
		-\frac{\bar\rho_s\bar U_{2s}}{\bar c_s^2}\tilde W_2
		\right),\\
&\tilde g_i(z')=r_i g_i^\sharp(z'),\,i=0,1,
	\end{split}
\end{eqnarray*}
and
\begin{eqnarray}\label{elip_coeff}
	\mathcal{A}_{ij}=
	\left(
	\begin{array}{ccc}
		z_1\bar\rho_s(1-\bar M_{1s}^2) &-\bar\rho_s \bar M_{1s}\bar M_{2s} &0\\
		-\bar\rho_s \bar M_{1s}\bar M_{2s} &\frac{1}{z_1}\bar\rho_s(1-\bar M_{2s}^2)&0\\
		0&0&z_1\bar\rho_s\\
	\end{array}
	\right),
		\mathcal{B}_{ij}=
	\left(
	\begin{array}{ccc}
		z_1 &0 &0\\
		0 &\frac{1}{z_1}&0\\
		0&0&z_1\\
	\end{array}
	\right).
\end{eqnarray}
A direct computation gives $\det|\mathcal{A}_{ij}(z_1)|=z_1\bar\rho_s^2(1-\bar M_{1s}^2-\bar M_{2s}^2)>0$, and $\det|\mathcal{B}_{ij}(z_1)|=z_1$. That is to say, $\mathcal{A}_{ij},\,\mathcal{B}_{ij}$ are both strictly positive symmetric matrixes in $\Omega$, and there exist positive constants $\lambda_A,\lambda_B>0$ satisfying
\begin{eqnarray}\label{coeff_posi}
	\mathcal{A}_{ij}(z_1)\zeta_i\zeta_j\geq \lambda_A|\zeta|^2,\\
	\mathcal{B}_{ij}(z_1)\zeta_i\zeta_j\geq \lambda_B|\zeta|^2,
\end{eqnarray}
for all $0<r_0\leq z_1\leq r_1$ and $\zeta\in \mathbb{R}^3$.
The following lemma tells that the linear boundary value problem \eqref{psi_W6}-\eqref{BC_psi_W6} has a unique weak solution.

It can be verified easily that the function $\tilde {\bf F},\,\tilde G $ and $\tilde g_i,\,i=0,1$ satisfy the following compatibility
conditions
\begin{eqnarray}\label{cp_FGg}
	\begin{cases}
		\p_{x_3}(\tilde F_1,\tilde F_2)(r,\th,\pm 1)=0,\\
		(\tilde F_3,\p_{x_3}^2\tilde F_3)(r,\th,\pm 1)=0,\\
		\p_{x_3}\tilde G(r,\th,\pm 1)=0,\\
		\p_{x_3}(\tilde g_0,\tilde g_1)(r,\th,\pm 1)=0.
	\end{cases}
\end{eqnarray}

For the convenience,  we introduce the following function to homogenize the boundary data
\begin{eqnarray}
	\tilde W_6(z_1,z')=\epsilon\left(
	\frac{r_1-z_1}{r_1-r_0}\Phi_{en}(z')
	+\frac{z_1-r_0}{r_1-r_0}\Phi_{ex}(z')
	\right),
\end{eqnarray}
and set $\phi=W_6-\tilde W_6$, then the system \eqref{psi_W6}-\eqref{BC_psi_W6} is converted into
\begin{eqnarray}\label{psi_phi}
	\begin{cases}
		\p_{z_i}(\mathcal{A}_{ij}\p_{z_j}\psi)
		+\p_{z_1}\left(\frac{z_1\bar\rho_s\bar U_{1s}}{\bar c_s^2} \phi\right)
		+\p_{z_2}\left(\frac{\bar\rho_s\bar U_{2s}}{\bar c_s^2} \phi\right)
		=\div_z{\hat {\bf F}},\\
		\p_{z_i}(\mathcal{B}_{ij}\p_{z_j}\phi)+\frac{z_1\bar\rho_s \bar U_{1s}}{\bar c_s^2} \p_{z_1}\psi+ \frac{\bar\rho_s \bar U_{2s}}{\bar c_s^2} \p_{z_2}\psi-\frac{z_1\bar\rho_s}{\bar c_s^2}\phi=\hat G,
	\end{cases}
\end{eqnarray}
and
\begin{eqnarray}\label{BC_psi_phi}
	\begin{cases}
		\left(r_0\bar\rho_s(1-\bar M_{1s}^2)\p_{z_1}\psi-\frac{\bar\rho_s \bar U_{1s}\bar U_{2s}}{\bar c_s^2}\p_{z_2}\psi\right)(r_0,z')=\tilde g_0(z'),\\
		\left(r_1\bar\rho_s(1-\bar M_{1s}^2)\p_{z_1}\psi-\frac{\bar\rho_s \bar U_{1s}\bar U_{2s}}{\bar c_s^2}\p_{z_2}\psi\right)(r_1,z')=\tilde g_1(z'),\\
		\p_{z_3}\psi(z_1,z_2,\pm 1)=0,\\
		\phi(r_0,z')=0,\,
		\phi(r_1,z')=0,\\
		\p_{z_3}\phi(z_1,z_2,\pm 1)=0,
	\end{cases}
\end{eqnarray}
where
\begin{eqnarray}
	\begin{split}
		&\hat F_1=\tilde F_1-\frac{z_1\bar\rho_s\bar U_{1s}}{\bar c_s^2}\tilde W_6,\quad
		\hat F_2=\tilde F_2-\frac{\bar\rho_s\bar U_{2s}}{\bar c_s^2}\tilde W_6,\quad\hat F_3=\tilde F_3,\\
		&\hat G=\tilde G-\p_{z_i}(\mathcal{B}_{ij}\p_{z_j}\tilde W_6)+\frac{z_1\bar\rho_s}{\bar c_s^2}\tilde W_6.
	\end{split}
\end{eqnarray}
Then, we have the following results.
\begin{lemma}\label{lemma}
The linear boundary value problem \eqref{psi_phi}-\eqref{BC_psi_phi} admits a unique weak solution $(\psi,\phi)\in H^1(\Omega)\times H^1(\Omega)$ such that
\begin{equation}\label{weak}
	\begin{split}
		\|(\psi,\phi)\|_{H^1(\Omega)}\leq& C
	\bigg(
	\|(\hat{\bf F},\hat G)\|_{L^2(\Omega)}+
	\|\tilde g_0-\hat F_1\|_{L^2(\Gamma^0)}\\
	&+
	\|\tilde g_1-\hat F_1\|_{L^2(\Gamma^1)}+
	\|\hat F_3\|_{L^2(\Gamma^{w\pm})}
	\bigg),
	\end{split}
\end{equation}
for some positive constant $C>0$.
\end{lemma}
\begin{proof}
To begin with, we introduce the following Hilbert space $\mathcal{H}$ in $\Omega$ as
\begin{equation*}
	\mathcal{H}=\{(\eta,\xi)\in H^1(\Omega)\times H^1(\Omega):\int _{\Omega}\eta dz=0,\,\xi(r_i,z')=0,i=0,1\}.
\end{equation*}
Hence, we say that $(\psi,\phi)\in \mathcal{H}$ is a weak solution to \eqref{psi_phi}-\eqref{BC_psi_phi} if and only if
\begin{equation}\label{def_weak_sol}
	\mathcal{B}[(\psi,\phi),(\eta,\xi)]=\mathcal{L}(\eta,\xi)
\end{equation}
is valid, where the bilinear operator $\mathcal{B}[\cdot,\cdot]$ and linear operator $\mathcal{L}$ on the Hilbert space $\mathcal{H}$ are respectively given by
\begin{eqnarray*}
	\begin{split}
		\mathcal{B}[(\psi,\phi),(\eta,\xi)]&:=
		\int_{\Omega} \mathcal{A}_{ij}\p_{z_j}\psi\p_{z_i}\eta\, dz
		+\int_{\Omega}\frac{z_1\bar\rho_s\bar U_{1s}}{\bar c_s^2}\phi\p_{z_1}\eta\, dz
		+\int_{\Omega}\frac{\bar\rho_s\bar U_{2s}}{\bar c_s^2}\phi\p_{z_2}\eta\, dz\\
		&\quad+\int_{\Omega} \mathcal{B}_{ij}\p_{z_j}\phi\p_{z_i}\xi dz
		-\int_{\Omega} \frac{z_1\bar\rho_s\bar U_{1s}}{\bar c_s^2}\p_{z_1}\psi\xi \,dz
		-\int_{\Omega}\frac{\bar\rho_s \bar U_{2s}}{\bar c_s^2}\p_{z_2}\psi\xi
		+\int_{\Omega}\frac{z_1\bar\rho_s}{\bar c_s^2}\phi\xi\,dz,
	\end{split}
\end{eqnarray*}
and
\begin{eqnarray*}
	\begin{split}
		\mathcal{L}(\eta,\xi)&:
		=\int_{\Omega}\left(\hat F_i\p_{z_i}\eta-\hat G\xi\right)\,dz+
		\int_{\{z_1=r_1\}} (\tilde g_1-\hat F_1)\,\eta dz_2dz_3\\
		&\quad-\int_{\{z_1=r_0\}} (\tilde g_0-\hat F_1)\,\eta dz_2dz_3
		-\int_{\{z_3=1\}}\hat F_3\eta dz_1dz_2
		+\int_{\{z_3=-1\}}\hat F_3\eta dz_1dz_2,
	\end{split}
\end{eqnarray*}
for any $(\eta,\xi)\in\mathcal{H}$.
Obviously, the bilinear form $\mathcal{B}[\cdot,\cdot]$ and the linear operator $\mathcal{L}$ are bounded in  $\mathcal{H}\times\mathcal{H}$ and $\mathcal{H}$, respectively. More importantly, the bilinear form $\mathcal{B}[\cdot,\cdot]$ is also coercive in $\mathcal{H}\times\mathcal{H}$ by virtue that
\begin{align*} \mathcal{B}[(\psi,\phi),(\psi,\phi)]&=\int_{\Omega}\mathcal{A}_{ij}\p_{z_j}\psi\p_{z_i}\psi+\mathcal{B}_{ij}\p_{z_j}\phi\p_{z_i}\phi+\frac{z_1\bar\rho_s}{\bar c_s^2}\phi^2 dz\\
    &\geq \int_{\Omega}\lambda_A|\nabla\psi|^2+\lambda_B|\nabla\phi|^2
	+\frac{z_1\bar\rho_s}{\bar c_s^2}\phi^2 dz\geq C\|(\psi,\phi)\|^2_{\mathcal{H}},
\end{align*}
for some positive constant $C$. It is worth mentioning that the mixed terms in $\mathcal{B}[(\psi,\phi),(\psi,\phi)]$ cancel each other and this special structure originally comes from the steady Euler-Poisson system and plays a key role in our analysis. Similar structure to the steady Euler-Poisson system in flat nozzles had been found in \cite{Bae1}.

By applying the Lax-Milgram theorem, one sees that there exists a unique weak solution $(\psi,\phi)\in \mathcal{H}$ to \eqref{psi_phi}-\eqref{BC_psi_phi}. In short, we can recover $W_6$ by $W_6=\phi+\tilde W_6$ and the proof is completed.
\end{proof}

Note that the system \eqref{psi_W6}-\eqref{BC_psi_W6} is similar to the elliptic equations investigated in \cite{Bae1}. By adapting the regularity theory developed in \cite{Bae1}, we can further improve the regularity of weak solution obtained in Lemma \ref{lemma}. Firstly, we can establish $C^\al$ estimate for the weak solution $(\psi,\phi)$.
\begin{lemma}\label{C_al_est}
There is a $\bar\al\in(0,1)$, such that for any $\al\in(0,\bar\al]$, the weak solutions $(\psi,\phi)$ obtained in Lemma \ref{lemma} satisfy
\begin{equation}\label{C_al}
	\|(\psi,\phi)\|_{C^{\al}(\bar\Omega)}\leq C_0,
\end{equation}
for some positive constant  $C_0$ depending only on the data and $\al$.
\end{lemma}

\begin{proof}
It follows from Theorem 3.1 in \cite{lh}, if there are constants $\ka_0,L_0>0$ satisfying
\begin{equation}
	\int_{B_r(z)\cap\Omega}|\nabla\psi|^2+|\nabla\phi|^2 dz\leq \ka_0^2 r^{1+2\al},
\end{equation}
for all $z\in\Omega$ and $r\in(0,L_0]$,
then one has
\begin{equation*}
	\|(\psi,\phi)\|_{C^{\al}(\bar\Omega)}\leq C
	\left(
	\ka_0+\|(\psi,\phi)\|_{L^2(\Omega)}
	\right).
\end{equation*}
Combining with \eqref{weak} gives the desired estimate \eqref{C_al}. For simplicity, we only consider the most subtle case where $z$ belongs to, without loss of generality, the junction of partial boundary $\Gamma^{w-}$ and $\Gamma^1$.

Fix $z^\star\in \bar\Gamma^{w-}\cap\bar \Gamma^1$ and set $\mathscr{D}_r(z^\star):=\Omega\cap B_r(z^\star)$, where $r>0$ is a constant with $r<\frac{1}{2}\min\{r_1-r_0,1\}$. Let $w_k,\,k=1,2$ be weak solutions of the following frozen coefficients linear problems:

\begin{equation*}
	\begin{cases}
				\begin{aligned}
			&\p_{z_i}(\mathcal{A}_{ij}(z^\star)\p_{z_j}w_1)=0,&\text{in}&\,\, \mathscr{D}_r(z^\star),\\
			&\left(\mathcal{A}_{ij}(z^\star)\p_{z_j}w_1\right)\cdot n_i^{out}=0, &\text{on}&\,\, \p\mathscr{D}_r(z^\star)\cap(\Gamma^{w-}\cap \Gamma^0),\\
						&w_1=\psi, &\text{on}&\,\, \p\mathscr{D}_r(z^\star)\cap\Omega,
		\end{aligned}
	\end{cases}
\end{equation*}
and
\begin{equation*}
	\begin{cases}
		\begin{aligned}
			&\p_{z_i}(\mathcal{B}_{ij}(z^\star)\p_{z_j}w_2)=0,&\text{in}&\,\, \mathscr{D}_r(z^\star),\\
			&\p_{{\bf n}^{out}}w_2=0, &\text{on}&\,\, \p\mathscr{D}_r(z^\star)\cap(\Gamma^{w-}\cap \Gamma^0),\\
			&w_2=\phi, &\text{on}&\,\, \p\mathscr{D}_r(z^\star)\cap\Omega,
		\end{aligned}
	\end{cases}
\end{equation*}
for the outward unit normal vector ${\bf n}^{out}=(n_1^{out},n_2^{out},n_3^{out})$ on $\p\mathscr{D}_r(z^\star)\cap(\Gamma^{w-}\cap\Gamma^0)$. That is to say, $(w_1,w_2)$ satisfy
\begin{equation}\label{def_wk_sol_w1w2}
	\int_{\mathscr{D}_r(z^\star)}
	\mathcal{A}_{ij}(z^\star)\p_{z_j}w_1\p_{z_i} \chi_1
	+\mathcal{B}_{ij}(z^\star)\p_{z_j}w_2\p_{z_i} \chi _2\,dz=0,
\end{equation}
for any test function $\chi_i\in H^1(\mathscr{D}_r(z^\star))\cap\{\chi_i=0\,\,\text{on}\,\,\p\mathscr{D}_r(z^\star)\cap\Omega\},\,i=1,2$. Naturally, we extend $\chi_i$ by setting $\chi_i=0$ in $\Omega\backslash \mathscr{D}_r(z^\star)$, and replacing $(\eta,\xi)$ in \eqref{def_weak_sol} by $(\chi_1,\chi_2)$. Now, subtracting \eqref{def_wk_sol_w1w2} from \eqref{def_weak_sol} gives
\begin{eqnarray}\label{psi_w_est}
\int_{\mathscr{D}_r(z^\star)} \mathcal{A}_{ij}(z^\star)\p_{z_j}(\psi-w_1)\p_{z_i}\chi_1
+\mathcal{B}_{ij}(z^\star)\p_{z_j}(\phi-w_2)\p_{z_i}\chi_2
\,dz=\mathcal{L}(\chi_1,\chi_2)+\sum_{k=1}^{4}I_k,
\end{eqnarray}
where $I_k,k=1,2,3,4$ are defined as
\begin{equation*}
	\begin{aligned}
I_1&=\int_{\mathscr{D}_r(z^\star)}
(\mathscr{A}_{ij}(z^\star)-\mathscr{A}_{ij}(z))
\p_{z_j}\psi\p_{z_i}\chi_1\,dz,\\
I_2&=\int_{\mathscr{D}_r(z^\star)}
\frac{z_1\bar\rho_s\bar U_{1s}}{\bar c_s^2}\phi\p_{z_1}\chi_1
+\frac{\bar\rho_s\bar U_{2s}}{\bar c_s^2}\phi\p_{z_2}\chi_1\,dz,\\
I_3&=\int_{\mathscr{D}_r(z^\star)}
(\mathscr{B}_{ij}(z^\star)-\mathscr{B}_{ij}(z))
\p_{z_j}\phi\p_{z_i}\chi_2\,dz,\\
I_4&=\int_{\mathscr{D}_r(z^\star)}
\left(
\frac{z_1\bar\rho_s}{\bar c_s^2}\phi
-\frac{z_1\bar\rho_s\bar U_{1s}}{\bar c_s^2}\p_{z_1}\psi
-\frac{\bar\rho_s \bar U_{2s}}{\bar c_s^2}\p_{z_2}\psi
\right)
\chi_2\,dz.
	\end{aligned}
\end{equation*}
Set $(\tilde\psi,\tilde\phi)=(\psi-w_1,\phi-w_2)$ and substitute $(\chi_1,\chi_2)=(\tilde\psi,\tilde\phi)$ into \eqref{psi_w_est} to get
\begin{equation}\label{C_al_key}
\int_{\mathscr{D}_r(z^\star)} \mathcal{A}_{ij}(z^\star)\p_{z_j}\tilde\psi\p_{z_i}\tilde\psi
+\mathcal{B}_{ij}(z^\star)\p_{z_j}\tilde\phi\p_{z_i}\tilde\phi
\,dz
\geq C
\int_{\mathscr{D}_r(z^\star)}
|\nabla \tilde\psi|^2+\\|\nabla\tilde\phi|^2\,dz,
\end{equation}
for some positive constant $C$.

Using H\"{o}lder inequality gives
\be\label{est_I1} \begin{split}
	&|I_1|\leq \|\mathscr{A}_{ij}\|_{C^1({\bar\Omega})}
	\left(
	r^2\int_{\mathscr{D}_r(z^\star)}|\nabla\psi|^2\,dz
	\right)^\frac{1}{2}
	\left(
	\int_{\mathscr{D}_r(z^\star)}|\nabla\tilde\psi|^2\,dz
	\right)^\frac{1}{2}\\
	&\quad\,\,\leq r \|\mathscr{A}_{ij}\|_{C^1({\bar\Omega})}
	\|\psi\|_{H^1(\Omega)}
		\left(
	\int_{\mathscr{D}_r(z^\star)}|\nabla\tilde\psi|^2\,dz
	\right)^\frac{1}{2}.
\end{split}  \ee
$I_3$ can be estimated similarly
\begin{eqnarray}\label{est_I3}
	|I_3|\leq\,r \|\mathscr{B}_{ij}\|_{C^1({\bar\Omega})}
	\|\phi\|_{H^1(\Omega)}
	\left(
	\int_{\mathscr{D}_r(z^\star)}|\nabla\tilde\phi|^2\,dz
	\right)^\frac{1}{2}.
\end{eqnarray}
$I_2$ is controlled by
\be\label{est_I2} \begin{split}
	&|I_2|\leq C
	\left(
	\int_{\mathscr{D}_r(z^\star)}|\phi|^2\,dz
	\right)^\frac{1}{2}
	\left(
	\int_{\mathscr{D}_r(z^\star)}|\nabla\tilde\psi|^2\,dz
	\right)^\frac{1}{2}\\
	&\quad\,\,\,\leq C r
	\left(
		\int_{\mathscr{D}_r(z^\star)}|\phi|^6\,dz
	\right)^\frac{1}{6}
	\left(
	\int_{\mathscr{D}_r(z^\star)}|\nabla\tilde\psi|^2\,dz
	\right)^\frac{1}{2}\\
	&\quad\,\,\,\leq C r
	\|\phi\|_{H^1(\Omega)}
	\left(
	\int_{\mathscr{D}_r(z^\star)}|\nabla\tilde\psi|^2\,dz
	\right)^\frac{1}{2}.
\end{split}  \ee
$I_4$ can be established by
\begin{eqnarray}\label{est_I4}
	|I_4|\leq C r
	\left(
	\|\phi\|_{H^1(\Omega)}
	+\|\psi\|_{H^1(\Omega)}
	\right)
	\left(
	\int_{\mathscr{D}_r(z^\star)}|\nabla\tilde\phi|^2\,dz
	\right)^\frac{1}{2},
\end{eqnarray}
where we have used Poincar\'{e} inequality.

In order to derive the estimation of the right-hand side of \eqref{psi_w_est}, it remains to study the linear functional $\mathcal{L}(\tilde\psi,\tilde\phi)$.
Indeed, it can be decomposed as
$$\mathcal{L}(\tilde\psi,\tilde\phi)= \mathcal{L}_1+\mathcal{L}_2+\mathcal{L}_3,$$
where
\begin{equation*}
	\begin{aligned}
\mathcal{L}_1&=\int_{\mathscr{D}_r(z^\star)}
\left(
\hat F_i\p_{z_i}\tilde\psi-\hat G\tilde\phi
\right)\,dz,\\
\mathcal{L}_2&=\int_{\p\mathscr{D}_r(z^\star)\cap\{z_1=r_1\}}(\tilde g_1-\hat F_1)\tilde\psi\,dz_2dz_3
,\\
\mathcal{L}_3&=\int_{\p\mathscr{D}_r(z^\star)\cap\{z_3=-1\}}\hat F_3\tilde\phi\,dz_1dz_2.
	\end{aligned}
\end{equation*}
It's easy to see
\begin{equation}\label{est_L1}
	|\mathcal{L}_1|\leq C r^\frac{3}{2}
	\left(
	\int_{\mathscr{D}_r(z^\star)}|\nabla\tilde\psi|^2+|\nabla\tilde\phi|^2\,dz
	\right)^\frac{1}{2}.
\end{equation}
By using trace theorem and Poincar\'{e} inequality with scaling, one can easily derive the estimates of $\mathcal{L}_k,k=2,3$:
\begin{equation}\label{est_L23}
	\begin{aligned}
		|\mathcal{L}_2|+|\mathcal{L}_3|&\leq Cr
		\left(
		\int_{\p\mathscr{D}_r(z^\star)}|\tilde\psi|^2
		\right)^\frac{1}{2}+
		Cr
		\left(
		\int_{\p\mathscr{D}_r(z^\star)}|\tilde\phi|^2
		\right)^\frac{1}{2}\\
		&\leq Cr^\frac{3}{2}
		\left(
		\int_{\mathscr{D}_r(z^\star)}|\nabla\tilde\psi|^2
		\right)^\frac{1}{2}
		+ Cr^\frac{3}{2}
		\left(
		\int_{\mathscr{D}_r(z^\star)}|\nabla\tilde\phi|^2
		\right)^\frac{1}{2}.
		\end{aligned}
\end{equation}
It follows from \eqref{C_al_key},\eqref{est_I1}-\eqref{est_I4} and \eqref{est_L1}-\eqref{est_L23}  that
\begin{equation}
	\int_{\mathscr{D}_r(z^\star)}
	|\nabla \tilde\psi|^2+\\|\nabla\tilde\phi|^2\,dz
	\leq C
	\left(
	\|(\psi,\phi)\|_{H^1(\Omega)}r^2+r^{3}
	\right)\leq C(C_H\,r^2 +r^3),
\end{equation}
where $C_H$ is a constant given by the right-hand side of \eqref{weak}.

 By adapting the proof of Lemma 4.12 in \cite{lh}, one can show that there exists two constants $\hat\al\in(0,1)$ and $\hat C>0$ depending only on the data such that $(w_1,w_2)$ satisfy
 \begin{equation}
 	\int_{\mathscr{D}_{\rho}(z^\star)}
 	|\nabla w_1|^2+|\nabla w_2|^2\,dz\leq \hat C(\frac{\rho}{r})^{1+2\hat\al}
 \int_{\mathscr{D}_{r}(z^\star)}|\nabla w_1|^2+|\nabla w_2|^2\,dz,
 \end{equation}
whenever $0<\rho\leq r$.

Hence, we have
 \begin{equation}
	\int_{\mathscr{D}_{\rho}(z^\star)}
	|\nabla \psi|^2+|\nabla \phi|^2\,dz\leq \hat C(\frac{\rho}{r})^{1+2\hat\al}
	\int_{\mathscr{D}_{r}(z^\star)}|\nabla \psi|^2+|\nabla \phi|^2\,dz+2C(C_H+1)^2r^2.
\end{equation}

Finally, for both the cases $1+2\hat\al\leq 2$ and $1+2\hat\al> 2$, by adapting the proof of Lemma 3.4 in \cite{lh}, one can show that there exists constants $\bar\al\in(0,\hat\al)$ and $L_0$ such that for any $r\in(0,L_0]$,
$(\psi,\phi)$ satisfy
\begin{equation}
	\int_{\mathscr{D}_r(z^\star)}|\nabla\psi|^2+|\nabla\phi|^2\,dz
\leq C r^{1+2\bar\al},
\end{equation}
for any $z^\star\in \bar\Gamma^{w-}\cap\bar \Gamma^1$. Thus, the proof is completed.
\end{proof}

In order to further establish uniform $C^{k,\al},k=1,2$ estimates of $(\psi,W_6)$. One can view the system \eqref{psi_W6} as two individual elliptic equations by leaving only the principal term on the left-hand side of the equations, and use symmetric extension to exclude the possible singularities that may appear at the circles $$\{(r,\th,x_3):r=r_0\,\,\text{or}\,\,r=r_1,\,\th\in T_{2\pi},\,x_3=\pm 1\}.$$
This is possible since that the principal parts of \eqref{psi_W6} are decoupled. For this, we introduce the enlarged domain $\Omega_e,\Gamma^0_e,\Gamma^1_e$ as
\begin{equation}\label{yantuo}
	\begin{aligned}
		\Omega_e&=\{(r,\th,x_3):r_0<r<r_1,\th\in\mathbb{T}_{2\pi},-3<x_3<3\},\\
		\Gamma_e^{0}&=\{(r,\th,x_3):r=r_0,\th\in\mathbb{T}_{2\pi},-3<x_3<3\},\\
		\Gamma_e^{1}&=\{(r,\th,x_3):r=r_1,\th\in\mathbb{T}_{2\pi},-3<x_3<3\},
	\end{aligned}
\end{equation}
and define the extended functions  on $\Omega_e$ as
\begin{equation}\label{exten}
	\begin{aligned}
		\begin{split}
			(U_1^e, &U_2^e, U_3^e, A^e, K^e, \Phi^e)\\
			=&\begin{cases}
				(U_1,U_2,U_3,A,K,\Phi)(r,\theta,x_3),\,\,
				(r,\theta,x_3)\in[r_0,r_1]\times[0,2\pi)\times[-1,1],\\
				(U_1,U_2,-U_3,A,K,\Phi)(r,\theta,-2-x_3),\,\,
				(r,\theta,x_3)\in[r_0,r_1]\times[0,2\pi)\times[-3,-1],\\
				(U_1,U_2,-U_3,A,K,\Phi)(r,\theta,2-x_3),\,\,
				(r,\theta,x_3)\in[r_0,r_1]\times[0,2\pi)\times[1,3].
			\end{cases}
		\end{split}
	\end{aligned}
\end{equation}

We may also extend $(m_{en},\omega_{en},K_{en},A_{en},\Phi_{en})$, $(m_{ex},\Phi_{ex})$ and $\tilde b$ to $\Gamma^0_{e},\Gamma^1_{e}$ and $\Omega_e$ respectively, which are still denoted by $(m_{en},\omega_{en},K_{en},A_{en},\Phi_{en},m_{ex},\Phi_{ex})$
and $\tilde b$. Thanks to the compatibility conditions \eqref{cp1} and \eqref{cp2}, $(U_1^e, U_2^e, U_3^e, A^e, K^e, \Phi^e)$ still satisfy the steady Euler-Poisson equations on the larger domain $\Omega_e$.

In the end, we obtain the following estimate
\begin{eqnarray}
	\begin{split}	
		\|(\psi, W_6)\|_{C^{2,\al}(\bar\Omega)}
		\leq C
		\bigg(
		&\sum_{i=1}^{3}\|\tilde F_i\|_{C^{1,\al}(\bar \Omega)}
		+\|\tilde G\|_{C^{\al}(\bar \Omega)}
		+\|g_0\|_{C^{1,\al}(\bar \Gamma^{0})}
		+\|g_1\|_{C^{1,\al}(\bar \Gamma^{1})}\\
		+&\epsilon\|\Phi_{en}\|_{C^{2,\al}(\bar \Gamma^{0})}
		+\epsilon\|\Phi_{ex}\|_{C^{2,\al}(\bar \Gamma^{1})}
		\bigg).
	\end{split}
\end{eqnarray}
For more details, one can refer to section 3.3 and 3.4 of \cite{Bae1}.

%Lemma \ref{lemma}, together with the H\"{o}lder inequality, the Sobolev inequality, the Poincar\'{e} inequality gives that
%\be\label{C0}
%||\psi||_{C^{\al}(\Omega)}+||W_6||_{C^{\al}(\Omega)}\leq C(||\tilde F_i||_{L^{\infty}(\Omega)}+||g_0||_{L^{\infty}(\Gamma^0)}+||g_1||_{L^{\infty}(\Gamma^1)}+||\tilde G||_{L^{\infty}(\Omega)}).
%\ee

As a consequence, we obtain a $C^{1,\al}$ smooth solution
\begin{equation*}
	W_1=\p_{z_1}\psi+\tilde W_1,\,
	W_2=\frac{1}{z_1}\p_{z_2}\psi+\tilde W_2,\,
	W_3=\p_{z_3}\psi+\tilde W_3,\,
	W_6=\phi+\tilde W_6,\,
\end{equation*}
for the system \eqref{def_curl}-\eqref{bc_def_curl}.

%\subsection{${\bf C^{2, \alpha}}$ Regularity}\label{subsec2}
We now proceed by setting $v_1=W_1,v_2=rW_2,v_3=W_3,v_4=W_6$, then the system \eqref{def_curl}-\eqref{bc_def_curl} can be rewritten as the following system
\begin{equation}\label{sys_v}
	\begin{cases}
		\sum _{i,j=1}^{3}\p_{z_i}(\mathcal{A}_{ij}v_j)
		+\p_{z_1}\left(\frac{z_1\bar\rho_s\bar U_{1s}}{\bar c_s^2} v_4\right)
		+\p_{z_2}\left(\frac{\bar\rho_s\bar U_{2s}}{\bar c_s^2} v_4\right)
		=f,\,\text{in}\,\Omega,\\
		\curl {\bf v}={\bf g},\,\text{in}\,\Omega,\\
		\sum _{i,j=1}^{3}\p_{z_i}(\mathcal{B}_{ij}\p_{z_j}v_4)
		+\frac{z_1\bar\rho_s \bar U_{1s}}{\bar c_s^2} v_1
		+ \frac{\bar\rho_s \bar U_{2s}}{\bar c_s^2} v_2
		-\frac{z_1\bar\rho_s}{\bar c_s^2}v_4=h,\,\text{in}\,\Omega,
	\end{cases}
\end{equation}
with boundary conditions
\begin{equation}\label{bc_v}
	\begin{cases}
		\sum_{j=1}^{3}\mathcal{A}_{1j}v_j \big|_{z_1=r_k}=b_k,\,k=0,1,\\
		v_3(z_1,z_2,\pm 1)=0,\\
		v_4(z_1,z_2,z_3)\big|_{z_1=r_k}=d_k,\,k=0,1,\\
		\p_{z_3}v_4(z_1,z_2,\pm 1)=0,
	\end{cases}
\end{equation}	
where
${\bf v}=(v_1,v_2,v_3)$, the matrix $\mathcal{A}_{ij},\,\mathcal{B}_{ij}$ is defined in \eqref{elip_coeff} satisfying \eqref{coeff_posi}. The vector field ${\bf g}$ satisfies the divergence free condition $\sum_{j=1}^{3}\p_{z_j}g_j=0$ in $\Omega$. In the following, we will verify some a priori estimates for the deformation-curl-Poisson system, the same estimates for the div-curl system is well-known, one can refer to \cite{abdg98, cs17, ky09}. Indeed, by introducing an enlarge system which is elliptic and the corresponding boundary conditions are complementing in the sense of Agmon-Dougalis-Nirenberg \cite{Nirenberg1}, we can further improve the regularity of the solution $(W_1, W_2, W_3, W_6)$.
\begin{lemma}\label{prior}
	Suppose $\al\in (0,1)$ and integer $m\geq 2$.
Let $({\bf v},v_4)\in (C^{1,\al}(\Omega))^3 \times C^{2,\al}(\Omega)$ be a classical solution to \eqref{sys_v}-\eqref{bc_v}, if $(f,g)\in C^{m-1,\al}(\Omega),\,h\in C^{m-2,\al}(\Omega)$, and $b_k,d_k\in C^{m,\al}(\Omega),\,k=0,1$, then $({\bf v},v_4)\in (C^{m,\al}(\Omega))^4$ satisfying the following estimates
	\begin{eqnarray}\label{est_v}
		\sum_{j=1}^{4}\|v_j\|_{C^{m,\al}(\Omega)}\leq C_2
		\left(
		\|(f,g)\|_{C^{m-1,\al}(\Omega)}+
		\|h\|_{C^{m-2,\al}(\Omega)}
		+\sum_{k=0}^{1}\|(b_k,d_k)\|_{C^{m,\al}(\p\Omega)}
		\right),
\end{eqnarray}
where $C_2$ is a positive constant.
\end{lemma}
\begin{proof}
In order to balance the number of the unknown functions and the equations in \eqref{sys_v}-\eqref{bc_v}, we consider an auxiliary system
\begin{equation}\label{aux_sys_v}
	\begin{cases}
		\sum _{i,j=1}^{3}\p_{z_i}(\mathcal{A}_{ij}v_j)
		+\p_{z_1}\left(\frac{z_1\bar\rho_s\bar U_{1s}}{\bar c_s^2} v_4\right)
		+\p_{z_2}\left(\frac{\bar\rho_s\bar U_{2s}}{\bar c_s^2} v_4\right)
		=f,\,\text{in}\,\,\Omega,\\
		\curl{\bf v}+\nabla N={\bf g},
		\,\text{in}\,\,\Omega,\\
		\sum _{i,j=1}^{3}\p_{z_i}(\mathcal{B}_{ij}\p_{z_j}v_4)
		+\frac{z_1\bar\rho_s \bar U_{1s}}{\bar c_s^2} v_1
		+ \frac{\bar\rho_s \bar U_{2s}}{\bar c_s^2} v_2
		-\frac{z_1\bar\rho_s}{\bar c_s^2}v_4=h,
		\,\text{in}\,\,\Omega,
	\end{cases}
\end{equation}
with boundary conditions
\begin{equation}\label{aux_bc_v}
	\begin{cases}
		\sum_{j=1}^{3}\mathcal{A}_{1j}v_j \big|_{z_1=r_k}=b_k,\,k=0,1,\\
		v_3(z_1,z_2,\pm 1)=0,\\
		v_4(z_1,z_2,z_3)\big|_{z_1=r_k}=d_k,\,k=0,1,\\
		\p_{z_3}v_4(z_1,z_2,\pm 1)=0,\\
		N(z_1,z_2,z_3)\big|_{z_1=r_k}=0,\,k=0,1,\\
		\p_{z_3}N(z_1,z_2,\pm 1)=0.
	\end{cases}
\end{equation}	
Since $\div {\bf g}=0$,\, we have $\Delta N=0$ in $\Omega$. Combining with the boundary conditions of $N$ on $\partial\Omega$, one can easily derive $N\equiv 0$ in $\Omega$. It follows that the system \eqref{aux_sys_v}-\eqref{aux_bc_v} is equivalent to \eqref{sys_v}-\eqref{bc_v}. Considering the unknown vector field $(v_1,v_2,v_3,v_4,v_5=N)$, we may take
\begin{equation*}
	\begin{split}
		\mu_1&=\mu_2=\mu_3=\mu_4=\mu_5=0,\\
		\nu_1&=\nu_2=\nu_3=\nu_5=1,\,\nu_4=2,
	\end{split}
\end{equation*}
where $\mu_i,\nu_j,\,i,j=1,2\cdots 5$ respectively represent the weights to the $i$-th equation and to the $j$-th unknown functions. Then, the characteristic matrix $\{l_{ij}(\xi)\}_{i,j=1}^{5}$ has the form
\begin{eqnarray}
	\{l_{ij}(\xi)\}_{i,j}=
\left(
\begin{array}{ccccc}
	\sum_{i=1}^{3}\mathcal{A}_{i1}\xi_i &
	\sum_{i=1}^{3}\mathcal{A}_{i2}\xi_i &
	\sum_{i=1}^{3}\mathcal{A}_{i3}\xi_i &
	0&0\\
	0&-\xi_3&\xi_2&0&\xi_1\\
	\xi_3&0&-\xi_1&0&\xi_2\\
	-\xi_2&\xi_1&0&0&\xi_3\\
	0&0&0&\sum_{i=1}^{3}\mathcal{B}_{ij}\xi_i\xi_j&0
\end{array}
\right)
\end{eqnarray}
with determinant $L(\xi)=|\xi|^2\sum_{i,j=1}^{3}\mathcal{A}_{ij}\xi_{i}\xi_{j}
\sum_{i,j=1}^{3}\mathcal{B}_{ij}\xi_{i}\xi_{j}\geq \lambda_A\lambda_B|\xi|^6$, which suggests that  the system \eqref{aux_sys_v} is elliptic in the sense of Douglis, Nirenberg \cite{Nirenberg1}.

In order to verify the  conditions are complementing, we let ${\bf V}(z)=e^{i(\xi_1z_1+\xi_2z_2)}{\bf W}(z_3)$, which solves the following system
\be\label{system47}
\begin{cases}
\sum_{j, k=1}^3\mathcal{A}_{jk}\p_{x_j}v_k=0, \quad\quad\, \mbox{in}\,\mathbb{R}_+^3,\\
\curl {\bf v}+\nabla N=0, \quad\quad\quad\quad\, \mbox{in}\,\mathbb{R}_+^3,\\
\Delta v_4=0, \quad\quad \quad\quad\quad\quad\quad\quad\mbox{in}\,\mathbb{R}_+^3,\\
(\mathcal{A}_{11}v_1+\mathcal{A}_{12}v_2+\mathcal{A}_{13}v_3)|_{z_1=r_k}=0, \quad\quad k=0,1,\\
N|_{z_1=r_k}=v_4|_{z_1=r_k}=0, \quad\quad\quad\quad\quad\quad\quad\,\, k=0,1,\\
(\p_{z_3}v_4, \p_{z_3}N)(z_1, z_2, \pm 1)=0,
\end{cases}
\ee
the next goal is to prove that ${\bf V}\rightarrow0$ as $z_3\rightarrow0$. Direct computations yields
\begin{equation} \label{system48}
\begin{cases}
i\xi_1\mathcal{A}_{11}w_1(z_3)+i\mathcal{A}_{12}(\xi_1w_2(z_3)+\xi_2w_1(z_3))+\mathcal{A}_{13}(i\xi_1w_3(z_3)+w_1'(z_3))\\
\quad\quad+\mathcal{A}_{23}(i\xi_2w_3(z_3)+w_2'(z_3))+i\mathcal{A}_{22}\xi_{2}w_2+\mathcal{A}_{33}w_3'(z_3)=0,\\
i\xi_2w_3(z_3)-w_2'(z_3)+i\xi_1w_5(z_3)=0,\\
w_1'(z_3)-i\xi_1w_3(z_3)+i\xi_2w_5(z_3)=0,\\
i\xi_1w_2(z_3)-i\xi_2w_1(z_3)+w_5'(z_3)=0,\\
-(\xi_1^2+\xi_2^2)w_4(z_3)+w_4''(z_3)=0,\\
(\mathcal{A}_{13}w_1+\mathcal{A}_{32}w_2+\mathcal{A}_{33}w_3)(0)=0,\\
w_4(0)=w_5(0)=0.
\end{cases}
\end{equation}
It is easy to derive
\be\label{system49}
\begin{cases}
w_1'(z_3)=i(\xi_1w_3-\xi_2w_5)(z_3),\\
w_2'(z_3)=i(\xi_2w_3+\xi_1w_5)(z_3).
\end{cases}
\ee
In terms of $i\xi_1w_2(z_3)-i\xi_2w_1(z_3)+w_5'(z_3)=0$, one can differentiate it with respect to $z_3$ to get
$$
w_5''(z_3)=(\xi_1^2+\xi_2^2)w_5(z_3)=0,
$$
which motivates that
$$
w_5(z_3)=C_1e^{-\sqrt{\xi_1^2+\xi_2^2}z_3}+C_2e^{\sqrt{\xi_1^2+\xi_2^2}z_3}.
$$
Since $w_5(0)=0$, and $w_5(z_3)\rightarrow0$ as $z_3\rightarrow0$, one has $C_1=C_2=0$, and $w_5\equiv0$. Similarly, $w_4\equiv0$. Then \eqref{system48} can be reduced to
\be\label{system410}
\begin{cases}
w_2'(z_3)=i\xi_2w_3(z_3),\\
w_1'(z_3)=i\xi_1w_3(z_3),\\
\xi_1w_2(z_3)=\xi_2w_1(z_3).
\end{cases}
\ee
Taking \eqref{system} into the first equation of \eqref{system48} and taking its derivative with respect to $z_3$, we obtain
\be\label{system412}
\mathcal{A}_{33}w_3''(z_3)+2i(\mathcal{A}_{13}\xi+\mathcal{A}_{23}\xi_2)w_3'(z_3)
-(\mathcal{A}_{11}\xi_1^2+2\mathcal{A}_{12}\xi_1\xi_2+\mathcal{A}_{22}\xi_2^2)w_3(z_3)=0.
\ee
Note that $\mathcal{A}_{ij}$ is positive symmetric, the real part of \eqref{system412} gives
$$
w_3''(z_3)-\frac{\mathcal{A}_{11}\xi_1^2+2\mathcal{A}_{12}\xi_1\xi_2+\mathcal{A}_{22}\xi_2^2}{\mathcal{A}_{33}}w_3(z_3)=0,
$$
from which we see
$$
w_3(z_3)=C_1 e^{-\sqrt{\frac{\Lambda}{\mathcal{A}_{33}}}z_3}+C_2 e^{\sqrt{\frac{\Lambda}{\mathcal{A}_{33}}}z_3},
$$
where
$ \Lambda= \mathcal{A}_{11}\xi_1^2+2\mathcal{A}_{12}\xi_1\xi_2+\mathcal{A}_{22}\xi_2^2$. Owing to $w_3(z_3)\rightarrow0$ as $z_3\rightarrow0$, we conclude that $C_2=0$. Then it follows \eqref{system48} that
\be\label{system414}
\begin{cases}
i\xi_1\mathcal{A}_{11}w_1(0)+i(2\mathcal{A}_{12}\xi_1+\mathcal{A}_{22}\xi_2)w_2(0)-\sqrt{\mathcal{A}_{33}\Lambda}w_3(0)=0,\\
-\xi_2w_1(0)+\xi_1w_2(0)=0,\\
\mathcal{A}_{13}w_1(0)+\mathcal{A}_{23}w_2(0)+\mathcal{A}_{33}w_3(0)=0.
\end{cases}
\ee
Since
\begin{align*}
\left|
\begin{array}{ccc}
i\xi\mathcal{A}_{11} & i(2\mathcal{A}_{12}\xi+\mathcal{A}_{22}\xi_2) & -\sqrt{\mathcal{A}_{33}\Lambda}\\
-\xi_2 & \xi_1 & 0\\
\mathcal{A}_{13} & \mathcal{A}_{23} & \mathcal{A}_{33}
\end{array}
\right|
=i\mathcal{A}_{33}\Lambda+(\mathcal{A}_{32}\xi_2+\mathcal{A}_{31}\xi_2)\sqrt{\mathcal{A}_{33}\Lambda}\neq0,
\end{align*}
then $w_j(0)=0$, $j=1,2 ,3$, which indicates $w_j(z_3)\equiv0$ for $j=1,2,3$. Hence the boundary conditions are complementing in the sense of \cite{Nirenberg1}. Then we have
\be\label{est_vv} \begin{split}
		&\sum_{j=1}^{4}\|v_j\|_{C^{m,\al}(\Omega)}\\
&\leq C_2
		\left(
		\|(f,g)\|_{C^{m-1,\al}(\Omega)}+
		\|h\|_{C^{m-2,\al}(\Omega)}
		+\sum_{k=0}^{1}\|(b_k,d_k)\|_{C^{m,\al}(\p\Omega)}+\sum_{j=1}^4||v_j||_{C^0(\Omega)}
		\right).
\end{split} \ee
By a standard contradiction argument and the similar uniqueness of system \eqref{sys_v} and \eqref{bc_v}, one can eliminate the zeroth order term in \eqref{est_vv}. Then the uniqueness can be verified similarly as in lemma \ref{lemma}. Hence, the proof of Lemma \ref{prior} is completed.
\end{proof}

By establishing the interior estimates on $\Omega_e$ with the aid of Lemma \ref{prior}, we obtain the following global estimates on $\Omega$ :
\begin{align*}
		\sum_{j=1}^{3}&\|W_j\|_{C^{2,\al}(\bar\Omega)}
		+\|W_6\|_{C^{2,\al}(\bar\Omega)}\\
&\leq C
	\bigg(
	\sum_{i=1}^{3}\|F_i\|_{C^{2,\al}(\bar\Omega)}
	+\sum_{i=4}^5||W_i||_{C^{2, \alpha}(\bar\Omega)}
	+\sum_{i=1}^3||\omega_i||_{C^{1, \alpha}(\bar\Omega)}\\
	&\quad+\|G\|_{C^{\al}(\bar\Omega)}
	+\|g_0\|_{C^{1,\al}(\bar \Gamma_{0})}
	+\|g_1\|_{C^{1,\al}(\bar \Gamma_{1})}
	+\epsilon\|\Phi_{en}\|_{C^{2,\al}(\bar \Gamma_{0})}
	+\epsilon\|\Phi_{ex}\|_{C^{2,\al}(\bar \Gamma_{1})}
	\bigg)\\
	&\leq C(\epsilon+\delta^2+\epsilon\delta).
\end{align*}

Up to now, for any fixed ${\bf W}^\sharp\in \mathcal{S}_{\delta}$, we have constructed a new vector ${\bf W}\in C^{2,\al}(\bar\Omega)$ with
\begin{equation}
	\sum_{j=1}^{6}\|W_j\|_{C^{2,\al}(\bar \Omega)}\leq C_*(\epsilon+\delta^2+\epsilon\delta).
\end{equation}

It remains to check the compatibility conditions of $(W_1,W_2,W_3,W_6)$. The boundary conditions \eqref{bc_def_curl}  imply that
\begin{equation}\label{cp_W3W6}
	W_3(r,\th,\pm 1)=\p_{x_3}W_6(r,\th,\pm 1)=0.
\end{equation}
The second and third equations in \eqref{def_curl} combined with the compatibility conditions of vorticity \eqref{cp_omega1}, \eqref{cp_omega23} leads to the compatibility conditions for $W_1,W_2$
\begin{equation}\label{cp_W1W2}
	\p_{x_3}(W_1,W_2)(r,\th,\pm 1)=0.
\end{equation}
Differentiating the first equation in \eqref{def_curl} with respect to $x_3$, and restricting the resulting equations on $\Gamma^{w\pm}$, we get
\begin{equation}\label{cp_W3}
	\p_{x_3}^2W_3(r,\th,\pm 1)=0,
\end{equation}
where we have used \eqref{cp_W4W5},\eqref{cp_tilde123},\eqref{cp_FGg} and \eqref{cp_W3W6}-\eqref{cp_W1W2}.

Selecting $\delta=2 C_*\epsilon $ so that
\begin{equation}
	\|{\bf W}\|_{C^{2,\al}(\bar \Omega)}\leq C_*\epsilon +C_*(2 C_*+1)\epsilon\delta<\delta,
\end{equation}
for any $\epsilon<\epsilon_*:=\frac{1}{2C_*(1+2C_*)}$. Moreover, $W_k,k=1,2,\cdots6$ satisfy the corresponding compatibility conditions and ${\bf W}\in \mathcal{S}_\delta$. Next, it is natural for us to construct a mapping $\mathcal{T}:\mathcal{S}_\delta\rightarrow\mathcal{S}_{\delta}$ by
$$\mathcal{T} {\bf W}^\sharp={\bf W}.$$
It remains to confirm that the mapping $\mathcal{T}$ is a contraction in $\mathcal{S}_{\delta}$. To this end, one may set
$${\bf W}^{\sharp, k}=(W_{1}^{\sharp, k}, W_{2}^{\sharp, k}, W_{3}^{\sharp, k}, W_{4}^{\sharp, k}, W_{5}^{\sharp, k}, W_{6}^{\sharp, k}),\,k=1,2,$$
from which we see ${\bf W}^k=\mathcal{T}{\bf W}^{\sharp, k}$. Define
\begin{equation*}
	{\bf Y}^\sharp={\bf W}^{\sharp, 1}-{\bf W}^{\sharp, 2},\,
	{\bf Y}={\bf W}^1-{\bf W}^2.
\end{equation*}

We next undertake to estimate ${\bf Y}$. In view of \eqref{eq_W4} and \eqref{eq_W5}, we first derive the equations for $Y_4$ and $Y_5$:
\begin{eqnarray}\label{eq_Y4}
	\begin{cases}
      \left(\p_r
      +\frac{W_{2}^{\sharp, 1}+\bar U_2}{W_{1}^{\sharp, 1}+\bar U_1}\frac{1}{r}\p_{\theta}
		+\frac{W_{3}^{\sharp, 1}}{W_{1}^{\sharp, 1}+\bar U_1}\p_{x_3}
		\right)Y_4\\
\quad\quad\quad\quad\quad\quad=-(R_2^{\sharp, 1}-R_2^{\sharp, 2})\frac{1}{r}\p_{\theta}W_{4}^2
-(R_3^{\sharp, 1}-R_3^{\sharp, 2})\p_{x_3}W_{4}^2,\\
		Y_4(r_0,\theta,x_3)=0,
	\end{cases}
\end{eqnarray}
and
\begin{eqnarray}\label{eq_Y5}
	\begin{cases}
		\left(\p_r
		+\frac{W_{2}^{\sharp, 1}+\bar U_{2s}}{W_{1}^{\sharp, 1}+\bar U_{1s}}\frac{1}{r}\p_{\theta}
		+\frac{W_{3}^{\sharp,1}}{W_{1}^{\sharp,1}+\bar U_{1s}}\p_{x_3}
		\right)Y_5\\
\quad\quad\quad\quad\quad\quad=-(R_2^{\sharp, 1}-R_2^{\sharp, 2})\frac{1}{r}\p_{\theta}W_{5}^2
-(R_3^{\sharp, 1}-R_3^{\sharp, 2})\p_{x_3}W_{5}^2,\\
		Y_5(r_0,\theta,x_3)=0,
	\end{cases}
\end{eqnarray}
where
\begin{eqnarray*}
R_j^{\sharp, k}=R_j^\sharp(r,\th,x_3;{\bf W}^{\sharp,k}),k=1,2,j=1,2,3.
\end{eqnarray*}
One can adopt the characteristic method to deduce that
\begin{equation}\label{est_diff_W4W5} \begin{split}
	\|(Y_4,Y_5)\|_{C^{1,\al}(\bar \Omega)}&\leq C\|{\bf W}^2\|_{C^{2,\al}(\bar \Omega)}
	\sum_{j=1}^{3}\|{\bf Y}^\sharp\|_{C^{1,\al}(\bar \Omega)}\\
&\leq C\delta||{\bf Y}^{\sharp}||_{C^{1,\al}(\bar \Omega)}.
\end{split} \end{equation}

We estimate the vorticity in the similar way. Denote ${\bf V}^k=(W_{1}^k,W_{2}^k,W_{3}^k),\,k=1,2$, and let
$\W_j=\omega_{j}^1-\omega_{j}^2,\,j=1,2,3$ with
\begin{eqnarray}\label{eq_diff_omg23}
\begin{split}
\W_2&=\frac{1}{W_{1}^{\sharp,1}+\bar U_{1s}}\left((W_{2}^{\sharp,1}+\bar U_{2s})\W_1
+\p_{x_3}Y_5
-\frac{H_{\sharp1}^{\gamma-1}}{\gamma-1}\p_{x_3}Y_4\right)\\
&\quad+(R_2^{\sharp, 1}-R_2^{\sharp, 2})\omega_{1}^2
+(R_1^{\sharp, 1}-R_1^{\sharp, 2})\p_{x_3}W_{5}^2
-\frac{1}{\gamma-1}(R_1^{\sharp, 1}H_{\sharp 1}-R_1^{\sharp, 2}H_{\sharp2})\p_{x_3}W_{4}^2,\\
\W_3&=\frac{1}{W_{1,1}^\sharp+\bar U_1}\left(W_{3}^{\sharp,1}\W_1-\frac{1}{r}\p_{\theta}Y_{5}^2
+\frac{H_{\sharp 1}^{\gamma-1}}{\gamma-1}\frac{1}{r}\p_{\theta}W_{4}^2\right)\\
&\quad+(R_3^{\sharp, 1}-R_3^{\sharp, 2})\omega_{1}^2
-(R_1^{\sharp, 1}-R_1^{\sharp, 2})\frac{1}{r}\p_{\theta}W_{5}^2
+\frac{1}{\gamma-1}(R_1^{\sharp, 1}H_{\sharp1}-R_1^{\sharp, 2}H_{\sharp2})\frac{1}{r}\p_{\theta}W_{4}^2,
\end{split}	
\end{eqnarray}
here $\omega_{j}^k$ is the $j$-th component of the curl of the vector field ${\bf V}^{k}$.
%and
%\begin{eqnarray*}
%	\begin{split}
%			&a_0^{\sharp i}(r,\theta,x_3;{\bf V}^\sharp_i),\quad\quad\quad\quad\,\quad\quad\quad\quad\quad\quad i=1,2,\\
%		&F_0^{\sharp i}=F_0^\sharp(r,\theta,x_3;{\bf V}^\sharp,W_4,W_5,W_6^\sharp),\quad\quad i=1,2.
%	\end{split}
%\end{eqnarray*}
Then the simple calculus attains that
\begin{equation}\label{eq_diff_omg1}
	\begin{cases}
				\left(
		\p_r
		+\frac{W_2^{\sharp,1}+\bar U_{2s}}{W_1^{\sharp, 1}+\bar U_{1s}}\frac{1}{r}\p_{\theta}
		+\frac{W_3^{\sharp,1}}{W_1^{\sharp,1}+\bar U_{1s}}\p_{x_3}
		\right)\W_1+
		a_0^\sharp(r,\theta,x_3;{\bf V}^{\sharp,1})\W_1\\
		=
		-R_2({\bf Y}^\sharp)\frac{1}{r}\p_{\theta}\omega_{1}^2
		-R_3({\bf Y}^\sharp)\p_{x_3}\omega_{1}^2
		-(a_{0}^{\sharp1}-a_{0}^{\sharp2})\omega_{1}^2
		+F_{0}^{\sharp1}-F_{0}^{\sharp2},\\
		\W_1(r_0,\theta,x_3)=0,
	\end{cases}
\end{equation}
where
\begin{eqnarray*}
H_{\sharp k}=
H(r,\theta,x_3;W_1^{\sharp, k},
W_2^{\sharp,k},W_3^{\sharp,k},W_{4}^k,W_{5}^k,W_6^{\sharp,k}),\quad k=1,2.
\end{eqnarray*}
Utilizing the characteristic method again gives that
\begin{eqnarray}
	\begin{split}
		\sum_{j=1}^{3}\|\W_j\|_{C^{\al}(\bar\Omega)}&\leq
		C
		\bigg(
		\|{\bf Y}^\sharp\|_{C^{1,\al}(\bar\Omega)}
		\|(W_{4}^2,W_{5}^2,{\bm\omega}^2)\|_{C^{1,\al}(\bar\Omega)}\\
		&\quad+\|(Y_4,Y_5)\|_{C^{1,\al}(\bar\Omega)}
		\|{\bf W}^\sharp\|_{C^{1,\al}(\bar\Omega)}
		\bigg)\\
		&\leq C\delta \|{\bf Y}^\sharp\|_{C^{1,\al}(\bar\Omega)}.
	\end{split}
\end{eqnarray}

The final task is to estimate $Y_j,j=1,2,3$ and $Y_6$. By virtue of \eqref{def_curl}-\eqref{bc_def_curl}, we obtain
\begin{eqnarray}\label{diff_def_curl}
	\begin{cases}
		\p_{r}(r\bar\rho_s(1-\bar M_{1s}^2)Y_1)
		-\p_{r}\left(r\bar\rho_s\bar M_{1s}\bar M_{2s} Y_2\right)
		+\p_{\theta}(\bar\rho_s(1-\bar M_{2s}^2)Y_2)\\
		\quad-\p_{\theta}\left(\bar\rho_s\bar M_{1s}\bar M_{2s}Y_1\right)
		+\p_{x_3}(r\bar\rho_s Y_3)
		+\p_{r}\left(\frac{r\bar\rho_s \bar U_{1s}}{\bar c_s^2}Y_6\right)
		+\p_{\theta}\left(\frac{\bar\rho_s \bar U_{2s}}{\bar c_s^2}Y_6\right)\\
		=
		\p_{r}(r\bar U_{1s} Q(Y_4,Y_5))			
		+\p_{\theta}(\bar U_{2s} Q(Y_4,Y_5))\\
		\quad+
		\p_{r}(r(F_1^{\sharp1}-F_1^{\sharp2}))
		+\p_{\theta}(F_2^{\sharp1}-F_2^{\sharp2})
		+\p_{x_3}(r(F_3^{\sharp1}-F_3^{\sharp2})),\\
		\frac{1}{r}\p_{\theta} Y_3-\p_{x_3} Y_2=\W_1,\\
		\p_{x_3} Y_1-\p_{r} Y_3=\W_2,\\
		\frac{1}{r}\p_{r}(rY_2)-\frac{1}{r}\p_{\theta} Y_1=\W_3,\\
		\Delta Y_6+\frac{\bar\rho_s \bar U_{1s}}{\bar c_s^2} Y_1+ \frac{\bar\rho_s \bar U_{2s}}{\bar c_s^2} Y_2-\frac{\bar\rho_s}{\bar c_s^2}Y_6=-Q(Y_4,Y_5)+G^{\sharp1}-G^{\sharp2},
	\end{cases}
\end{eqnarray}
supplemented with the following boundary conditions
\begin{eqnarray}\label{bc_diff_def_curl}
	\begin{cases}
		\left(\bar\rho_s(1-\bar M_{1s}^2) Y_1-\bar\rho_s \bar M_{1s}\bar M_{2s}Y_2\right)(r_0, \th, x_3)=(g_0^{\sharp1}-g_0^{\sharp2})(\theta,x_3),\\
		\left(\bar\rho_s(1-\bar M_{1s}^2) Y_1-\bar\rho_s \bar M_{1s}\bar M_{2s}Y_2\right)(r_1, \th, x_3)=(g_1^{\sharp1}-g_1^{\sharp2})(\theta,x_3)+\ka^1-\ka^2,\\
		Y_3(r,\theta,\pm 1)=0,\\
		Y_6(r_0,\theta,x_3)=
		Y_6(r_1,\theta,x_3)=0,\\
		\p_{x_3}Y_6(r,\theta,\pm 1)=0,
	\end{cases}
\end{eqnarray}
where
\begin{eqnarray*}
	\begin{split}
&F_j^{\sharp k}=F_j(r,\theta,x_3;{\bf V}^{\sharp, k},W_6^{\sharp, k}),\quad k=1,2,\,j=1,2,3,\\
&G^{\sharp k}=G(r,\theta,x_3;{\bf V}^{\sharp, k},W_6^{\sharp, k}),\quad\,
k=1,2,\\
&g_{j}^{\sharp k}=g_{j}^\sharp(\theta,x_3;{\bf V}^{\sharp, k}),\quad\quad\quad\,\,\quad\,\,\, k=1,2, j=0,1.
	\end{split}
\end{eqnarray*}
Consequently, we have the following estimates
\begin{equation}
\begin{split}
\sum_{j=1}^{3}&\|Y_j\|_{C^{1,\al}(\bar \Omega)}
+\|Y_6\|_{C^{1,\al}(\bar\Omega)}\\
&\leq C\bigg(
\sum_{j=1}^{3}\|F_j^{\sharp1}-F_j^{\sharp2}\|_{C^{1,\al}(\bar \Omega)}
+\|(Y_4,Y_5)\|_{C^{1,\al}(\bar \Omega)}
+\sum_{j=1}^{3}\|\W_j\|_{C^{\al}(\bar \Omega)}\\
&\quad\quad+\|G^{\sharp1}-G^{\sharp2}\|_{C^{1,\al}(\bar \Omega)}
+\sum_{j=0}^{1}\|g_j^{\sharp1}-g_j^{\sharp2}\|_{C^{1,\al}(\bar \Omega)}
+\|\ka^1-\ka^2\|_{C^{1,\al}(\bar \Omega)}
		\bigg)\\
		&\leq C\delta\|{\bf Y}^\sharp\|_{C^{1,\al}(\bar \Omega)}.
\end{split}
\end{equation}
Collecting all those estimates together, we achieve
\begin{equation}
	\|{\bf Y}\|_{C^{1,\al}(\bar \Omega)}
	\leq C_{**} \delta
	\|{\bf Y}^\sharp\|_{C^{1,\al}(\bar \Omega)}.
\end{equation}
The contraction of the mapping $\mathcal{T}$ with regard to $C^{1, \alpha}$ norm follows for any $\epsilon\leq\epsilon_0$ through setting $C_{**}\delta=3C_*C_{**}\epsilon\leq\frac2 3$ and $\epsilon_0=\min\{\epsilon_*, \frac{1}{3C_*C_{**}}\}$, which further illustrates that there exists a unique fixed point ${\bf W}\in\mathcal{S}_{\delta}$ of $\mathcal{T}$. Therefore, the velocity field and the density can be recovered by
$${\bf U}=(W_1+\bar U_{1s},W_2+\bar U_{2s},W_3),$$
and
$$\rho=H(W_4+A_0, W_5+K_0, W_6+\bar\Phi_s, |\U|^2).$$
Moreover, the equations for \eqref{def_curl}-\eqref{bc_def_curl} are changed back to
\begin{eqnarray*}
	\begin{cases}
\frac{1}{r}\p_r(r\rho U_1)+\frac{1}{r}\p_{\th}(\rho U_2)+\p_{x_3}(\rho U_3)=0,\\
		\frac{1}{r}\p_{\theta} U_3-\p_{x_3} U_2=\omega_1,\\
		\p_{x_3} U_1-\p_{r} U_3=\omega_2,\\
		\frac{1}{r}\p_{r}(rU_2)-\frac{1}{r}\p_{\theta} U_1=\omega_3,\\
		\Delta \Phi=\rho-b,
	\end{cases}
\end{eqnarray*}
and
\begin{eqnarray*}
	\begin{cases}
		(\rho U_1)(r_0,\theta,x_3)=(\bar\rho_s\bar U_{1s})(r_0)+\epsilon m_{en}(\theta,x_3),\\
		(\rho U_1)(r_1,\theta,x_3)=(\bar\rho_s\bar U_{1s})(r_1)+\epsilon m_{en}(\theta,x_3)+\ka,\\
		U_3(r,\th,\pm 1)=0,\\
	\Phi(r_0,\th,x_3)=\epsilon\Phi_{en}(\th,x_3),\\
	\Phi(r_1,\th,x_3)=\epsilon\Phi_{ex}(\th,x_3),\\
	\p_{x_3}\Phi(r,\th,\pm 1)=0.\\
	\end{cases}
\end{eqnarray*}
The compatibility condition \eqref{cp3} suggests that $\ka=0$. Hence, the proof is finished.

\begin{remark}
The proof of Theorem \ref{theoremm} is similar to the proof of Theorem \ref{theorem}. It is worth mentioning that the Bernoulli's function, the entropy and the vorticity field can be resolved as above, while the handle of the velocity and the electrostatic potential are different from the Theorem \ref{theorem}. And there is no need to introduce a constant $\ka$ by noting that there are no compatibility conditions. One may refer to section 3 of \cite{Weng4} to acquire more details.
\end{remark}

%Lemma \ref{lemma}, together with the H\"{o}lder inequality, the Sobolev inequality, the Poincar\'{e} inequality gives that
%\be\label{C0}
%||\psi||_{C^{\al}(\Omega)}+||W_6||_{C^{\al}(\Omega)}\leq C(||\tilde F_i||_{L^{\infty}(\Omega)}+||g_0||_{L^{\infty}(\Gamma^0)}+||g_1||_{L^{\infty}(\Gamma^1)}+||\tilde G||_{L^{\infty}(\Omega)}).
%\ee

%\subsection{${\bf C^{2, \alpha}}$ Regularity}\label{subsec2}

\section{The structural stability of smooth transonic flows}

This section devotes to investigate the structural stability of smooth transonic flows with nonzero vorticity under suitable axi-symmetric perturbations.

Define
\be\label{quyu}
\mathbb{D}:=(r_0, r_1)\times(-1, 1),
\ee
and the boundary $\p\mathbb{D}$, which is denoted by $\Sigma^0, \Sigma^1, \Sigma^w$, becomes
\begin{align*}
&\Sigma^0=\{(r, x_3):r=r_0, -1<x_3<1\}, \\
&\Sigma^1=\{(r, x_3):r=r_1, -1<x_3<1\}, \\
&\Sigma^w=\{(r, x_3):r_0<r<r_1, x_3=\pm 1\}.
\end{align*}
We target on looking for axi-symmetric transonic flows in the concentric cylinder $\mathbb{D}$ satisfying suitable boundary conditions on the inner and outer cylinder respectively.

In this case, we assume that the velocity and the density are of the form
\be\label{115}
{\bf u}(x)=U_1(r, x_3){\bf e}_r+U_2(r, x_3){\bf e}_{\th}+U_3(r, x_3){\bf e}_3, \quad \rho(x)=\rho(r, x_3), \quad A(x)=A(r, x_3),
\ee
then the system \eqref{ProblemImm} is reduced to
\be\label{116}
\begin{cases}
\p_r(r\rho U_1)+\p_{x_3}(r\rho U_3)=0,\\
\rho(U_1\p_r+U_3\p_{x_3})U_1+\p_rP-\frac{\rho U_2^2}{r}=\rho\cdot\p_r\Phi,\\
\rho(U_1\p_r+U_3\p_{x_3})U_2+\frac{\rho U_1U_2}{r}=0,\\
\rho(U_1\p_r+U_3\p_{x_3})U_3+\p_{x_3}P=\rho\cdot\p_{x_3}\Phi,\\
\rho(U_1\p_r+U_3\p_{x_3})A=0,\\
(\p_r^2+\frac1 r\p_r+\p_{x_3}^2)\Phi=\rho-b.
\end{cases}
\ee

We prescribe the following boundary conditions on the inner cylinder:
\be\label{117}
\begin{cases}
U_2(r_0, x_3)=\bar U_{2t}(r_0)+\eps \tilde U_{2, en}(x_3),\\
U_3(r_0, x_3)=\eps \tilde U_{3, en}(x_3),\\
A(r_0, x_3)=\bar A_{t}+\eps\tilde{A}_{en}(x_3), \quad K(r_0, x_3)=\bar K_{t}+\eps\tilde{K}_{en}(x_3),\\
\Phi(r_0, x_3)=\bar\Phi_{t}(r_0)+\eps\tilde\Phi_{en}(x_3),
\end{cases}
\ee
with $\tilde U_{2, en}(x_3), \tilde U_{3, en}(x_3), \tilde A_{en}(x_3), \tilde K_{en}(x_3), \tilde\Phi_{en}(x_3)\in C^{2, \al}(\mathbb{D})$ and $\eps$ small enough. And the boundary conditions posed on the outer cylinder are
\be\label{118}
\begin{cases}
U_1(r_1, x_3)=\bar U_{1t}(r_1)+\eps \tilde U_{1, ex}(x_3),\\
\Phi(r_1, x_3)=\bar\Phi_{t}(r_1)+\eps\tilde\Phi_{ex}(x_3),
\end{cases}
\ee
with $\tilde U_{1, ex}(x_3), \tilde\Phi_{ex}(x_3)\in C^{2, \al}(\mathbb{D})$.
%and satisfy the following compatibility conditions:
%\be\label{cp22}
%q_1(x_3)=\p_{x_3}\tilde\Phi_{ex}(x_3)=0, \quad x_3=\pm 1.
%\ee
Also, the compatibility conditions are prescribed as
\be\label{cp111}
\begin{cases}
\tilde U_{3, en}(\pm 1)=\tilde U_{3, en}''(\pm1)=0,\\
\tilde U_{1, ex}'(\pm1)=\tilde U_{2, en}'(\pm 1)=\tilde K_{en}'(\pm 1)=\tilde A_{en}'(\pm 1)=0.
\end{cases}
\ee
Moreover, the slip boundary conditions on $x_3=\pm 1$ are given as:
\be\label{119}
U_3(r, \pm 1)=0, \,\, \p_{x_3}\Phi(r, \pm 1)=0.
\ee

The perturbation of the ion background density is also imposed as
\be\label{b}
b(r, x_3)=b_0+\eps\tilde b(r, x_3),
\ee
where $\tilde{b}\in C^{\al}(\bar{\mathbb{D}})$.

The following theorem on the existence and uniqueness of smooth axi-symmetric transonic spiral flows with nonzero vorticity holds.
\bt\label{th3}
Given any background flow with nonzero radial velocity $\bar U_{1}\neq0$, for any smooth $C^{2, \alpha}(\mathbb{D})$ functions $\tilde A_{en}, \tilde K_{en}, \tilde\Phi_{en}, \tilde\Phi_{ex}$ and $\tilde U_{1, ex}, \tilde U_{2, en}, \tilde U_{3, en}$, there exists a small constant $\eps_0$ depending only upon the background flow and the boundary datum, such that if $0<\eps\leq\eps_0$, there exists a unique smooth transonic flow with nonzero vorticity
$$
{\bf u}=U_1(r, x_3){\bf e}_r+U_2(r, x_3){\bf e}_{\th}+U_3(r, x_3){\bf e}_3, \, A(x)=A(r, x_3), \, K(x)=K(r, x_3), \,\Phi(x)=\Phi(r, x_3)
$$
to \eqref{116} and \eqref{117} and \eqref{118}, with the following estimate holds
\be\label{120}
\sum\limits_{j=1}^{2}||U_j-\bar U_{jt}||_{C^{2, \al}(\bar{\mathbb{D}})}+||U_3||_{C^{2, \al}(\bar{\mathbb{D}})}+||K-\bar K_t||_{C^{2, \al}(\bar{\mathbb{D}})}+||A-\bar A_t||_{C^{2, \al}(\bar{\mathbb{D}})}+||\Phi-\bar\Phi_t||_{C^{2, \al}(\bar{\mathbb{D}})}\leq C_0\eps,
\ee
for some constant $C_0$ depending only on the background solution and the boundary datum.
\et

\begin{remark}
The structural stability of the transonic solution under cylindrical perturbations as was done in \cite{Weng2} is much more complicated and will be investigated in the near future.
\end{remark}

By employing again the deformation-curl-Poisson decomposition obtained in Section \ref{sec1}, the system \eqref{116} is equivalent to
\be\label{41}
\begin{cases}
\p_r\left(rH(A, K, \Phi, |\U|^2)U_1\right)+\p_{x_3}\left(rH(A, K, \Phi, |\U|^2)U_3\right)=0,\\
(U_1\p_r+U_3\p_{x_3})(r U_2)=0,\\
(U_1\p_r+U_3\p_{x_3})A=0,\\
(U_1\p_r+U_3\p_{x_3})K=0,\\
U_1(\p_r U_3-\p_{x_3} U_1)=U_2\p_{x_3}U_2+\frac{1}{\ga-1}\rho^{\ga-1}\p_{x_3}A-\p_3K,\\
(\p_r^2+\frac1 r\p_r+\p_{x_3}^2)\Phi=\rho-b.
\end{cases}
\ee

We define the deviations between the flow and the background flow, which are still denoted same as before
\begin{align*}
&W_1=U_1-\bar U_{1t}, \quad W_2=U_2-\bar U_{2t}, \quad W_3=U_3, \\
&W_4=A-\bar A_t, \quad W_5=K-\bar K_t,\quad W_6=\Phi-\bar\Phi_t.
\end{align*}
The solution space $\mathcal{S}_{\delta_0}$ consists of the vector functions ${\bf W}=(W_1,\ldots, W_6)\in C^{2, \al}(\bar{\mathbb{D}})$ satisfying the estimate
$$
||{\bf W}||_{\mathcal{S}_{\delta_0}}:=\sum\limits_{j=1}^6||W_j||_{C^{2, \al}(\bar{\mathbb{D}})}\leq\delta_0,
$$
and the following compatibility conditions:
\be\label{compati}
\begin{cases}
(W_3, \p_{x_3}^2W_3)(r, \pm 1)=0, \quad\quad\quad\quad\quad\quad \mbox{on} \,\, r\in[r_0, r_1],\\
(\p_{x_3}W_2, \p_{x_3}W_4, \p_{x_3}W_5)(r, \pm1)=0, \quad\,\, \mbox{on} \,\, r\in[r_0, r_1],\\
(\p_{x_3}W_1, \p_{x_3}W_6)(r, \pm 1)=0, \quad\quad\quad\quad\,\,\, \mbox{on}\,\, r\in[r_0, r_1].
\end{cases}
\ee

For any ${\bf W}^\sharp\in\mathcal{S}_{\delta_0}$, we construct an operator $\mathcal{P}: {\bf W}^\sharp\in\mathcal{S}_{\delta_0}\mapsto{\bf W}\in\mathcal{S}_{\delta_0}$. The procedure of solving ${\bf W}$ will be divided into the following steps.

We first solve the following hyperbolic problems to obtain $(W_2, W_4, W_5)$:
\be\label{46}
\begin{cases}
((\bar U_{1t}+W_1^\sharp)\p_r+W_3^\sharp\p_{x_3})(rW_2, W_4, W_5)=0,\\
(W_2, W_4, W_5)(r_0, x_3)=(\eps \tilde U_{2, en}(x_3), \eps\tilde{A}_{en}(x_3), \eps\tilde{K}_{en}(x_3)).
\end{cases}
\ee
%By differentiating the first equation of \eqref{46} with respect to $x_3$, one can verify that $(W_2, W_4, W_5)$ satisfy the following compatibility conditions:
%\be\label{245compat}
%(\p_{x_3}W_2, \p_{x_3}W_4, \p_{x_3}W_5)(r, \pm 1)=0, \quad \mbox{on}\,\, r\in[r_0, r_1].
%\ee
Then the characteristics method gives
\be\label{47}
||(W_2, W_4, W_5)||_{C^{2, \al}(\bar{\mathbb{D}})}\leq C\eps,
\ee
where $C$ depends only upon the background solution and the boundary data. One can further verify the following compatibility conditions:
\be\label{245compat}
(\p_{x_3}W_2, \p_{x_3}W_4, \p_{x_3}W_5)(r, \pm 1)=0, \quad \mbox{on}\,\, r\in[r_0, r_1].
\ee

Next, with the aid of \eqref{midu}, one can replace $(\bar\rho_s, \bar U_{1s}, \bar U_{2s}, \bar\Phi_s)$ in \eqref{mimi} and \eqref{shengsu} by $(\bar\rho_t, \bar U_{1t}, \bar U_{2t}, \bar\Phi_t)$, to get
 \begin{equation}\label{mimiii} \begin{cases}
\rho=\bar\rho_t
-\frac{\bar\rho_t \bar U_{1t}}{\bar c_t^2} W_1-\frac{\bar\rho_t \bar U_{2t}}{\bar c_t^2} W_2-\frac{\bar\rho_t}{(\gamma-1)A_0}W_4+\frac{\bar\rho_t}{\bar c_t^2}W_5+\frac{\bar\rho_t}{\bar c_t^2}W_6+O(|{\bf W}|^2),\\
c^2=\bar c_t^2+(\gamma-1)\left(W_5+W_6-\bar U_{1t} W_1-\bar U_{2t} W_2\right)-\frac{1}{2}(\gamma-1)(W_1^2+W_2^2+W_3^2),
\end{cases}  \end{equation}
here
\begin{eqnarray*}\no
\bar\rho_t=H(A_0,K_0,\bar\Phi_t,|\bar \U_t|^2), \quad \bar c_t=c(\bar\rho_t, K_0, \bar\Phi_t).
\end{eqnarray*}
Plugging \eqref{mimiii} into \eqref{41}, then $(W_1, W_3, W_6)$ will be obtained by solving the following linear system:
\be\label{49}
\begin{cases}
\p_r\left(r\bar\rho_t(1-\bar M_{1t}^2)W_1\right)+\p_{x_3}(r\bar\rho_t W_3)+\p_r\left(\frac{r\bar\rho_t \bar U_{1t}}{\bar c_t^2}W_6\right)\\
\quad\quad\quad\quad=\p_r\left(r\bar\rho_t\bar M_{1t}\bar M_{2t}W_2\right)+\p_r\left(\frac{r\bar\rho_t\bar U_{1t}}{(\ga-1)A_0}W_4-\frac{r\bar\rho_t\bar U_{1t}}{\bar c_t^2}W_5\right)+\p_r(r G_1)+\p_{x_3}(r G_3),\\
\p_r W_3-\p_{x_3} W_1=G_2,\\
(\p_r^2+\frac1 r\p_r+\p_{x_3}^2)W_6+\frac{\bar\rho_t\bar U_{1t}}{\bar c_t^2}W_1-\frac{\bar\rho_t}{\bar c_t^2}W_6=-\frac{\bar\rho_t\bar U_{2t}}{\bar c_t^2}W_2-\frac{\bar\rho_t}{(\ga-1)A_0}W_4+\frac{\bar\rho_t}{\bar c_t^2}W_5+G_4,\\
W_1(r_1, x_3)=\eps \tilde U_{1, ex}(x_3),\\
W_3(r_0, x_3)=\eps\tilde U_{3, en}(x_3),\\
W_3(r, \pm 1)=0, \quad \forall r\in[r_0, r_1]\\
W_6(r_0, x_3)=\eps\tilde\Phi_{en}(x_3), \quad W_6(r_1, x_3)=\eps\tilde\Phi_{ex}(x_3), \quad \p_{x_3}W_6(r, \pm 1)=0, \,\,\forall r\in[r_0, r_1],
\end{cases}
\ee
where
\be\label{45}  \begin{split}
&G_1(W_1^\sharp, W_2, W_3^\sharp, W_4, W_5, W_6^\sharp)\\
&\quad=-\left(H-\bar H_t\right)W_1^\sharp-\Bigg(H-\bar H_t+\frac{\bar\rho_t}{(\ga-1)A_0}W_4-\frac{\bar\rho_t}{\bar c_t^2}W_5-\frac{\bar\rho_t}{\bar c_t^2}W_6^{\sharp}+\frac{\bar\rho_t\bar U_{1t}}{\bar c_t^2}W_1^\sharp+\frac{\bar\rho_t\bar U_{2t}}{\bar c_t^2}W_2\Bigg)\bar U_{1t},\\
&G_2(W_1^\sharp, W_2, W_3^\sharp, W_4, W_5, W_6^\sharp)=\frac{1}{\bar U_{1t}+W_1^\sharp}\left((\bar U_{2t}+W_2)\p_{x_3}W_2+\frac{H^{\ga-1}({\bf W}^\sharp)}{\ga-1}\p_{x_3}W_4-\p_{x_3}W_5\right),\\
&G_3(W_1^\sharp, W_2, W_3^\sharp, W_4, W_5, W_6^\sharp)=-(H-\bar H_t)W_3^\sharp,\\
&G_4(W_1^\sharp, W_2, W_3^\sharp, W_4, W_5, W_6^\sharp)\\
&\quad=H-\bar H_t+\frac{\bar\rho_t}{(\ga-1)A_0}W_4-\frac{\bar\rho_t}{\bar c_t^2}W_5-\frac{\bar\rho_t}{\bar c_t^2}W_6^{\sharp}+\frac{\bar\rho_t\bar U_{1t}}{\bar c_t^2}W_1^\sharp+\frac{\bar\rho_t\bar U_{2t}}{\bar c^2}W_2-\eps \tilde b,
\end{split} \ee
with
\begin{align*}
&\rho=H(W_4+A_0, W_5+K_0, W_6^\sharp+\Phi_0, |W_1^\sharp+\bar U_{1t}|^2,|W_2+\bar U_{2t}|^2, |W_3^{\sharp}|^2)=:H,\\
&\bar\rho_t=H(A_0, K_0, \bar\Phi_t, |\bar\U_t|^2)=:\bar H_t, \quad \bar c_t=c(\bar\rho_t, K_0, \bar\Phi_t).
\end{align*}

By virtue of \eqref{compati}, one can verify that $G_1, G_2, G_3, G_4$ satisfy the following compatibility conditions:
\be\label{Gcompat}  \begin{cases}
\p_{x_3}G_1(W_1^\sharp, W_2, W_3^\sharp, W_4, W_5, W_6^\sharp)(r, \pm 1)=0,\\
G_2(W_1^\sharp, W_2, W_3^\sharp, W_4, W_5, W_6^\sharp)(r, \pm 1)=0,\\
G_3(W_1^\sharp, W_2, W_3^\sharp, W_4, W_5, W_6^\sharp)(r, \pm 1)=0,\\
\p_{x_3}G_4(W_1^\sharp, W_2, W_3^\sharp, W_4, W_5, W_6^\sharp)(r, \pm 1)=0.
\end{cases} \ee

Thanks to \eqref{Gcompat} and the estimate \eqref{47}, one has
\be\label{GGG}
\begin{cases}
||G_1(W_1^\sharp, W_2, W_3^\sharp, W_4, W_5, W_6^\sharp)||_{C^{1, \al}(\bar{\mathbb{D}})}\leq C(\eps+\delta_0^2),\\
||G_2(W_1^\sharp, W_2, W_3^\sharp, W_4, W_5, W_6^\sharp)||_{C^{1, \al}(\bar{\mathbb{D}})}\leq C\eps,\\
||G_3(W_1^\sharp, W_2, W_3^\sharp, W_4, W_5, W_6^\sharp)||_{C^{1, \al}(\bar{\mathbb{D}})}\leq C(\delta_0+\eps)\delta_0,\\
||G_4(W_1^\sharp, W_2, W_3^\sharp, W_4, W_5, W_6^\sharp)||_{C^{1, \al}(\bar{\mathbb{D}})}\leq C(\eps+\delta_0^2).
\end{cases}
\ee
It is easy to see that there exists a unique solution $\psi_1(r, x_3)$ to the following problem
\be\label{411}
\begin{cases}
(\p_r^2+\p_{x_3}^2)\psi_1=G_2(W_1^\sharp, W_2, W_3^\sharp, W_4, W_5, W_6^\sharp),\\
\psi_1(r_1, x_3)=\p_r\psi_1(r_0, x_3)=0,\\
\psi_1(r, \pm 1)=0.
\end{cases}
\ee

To deal with the singularity near the corner, we employ the standard symmetric extension technique again to extend $\mathbb{D}$ along the $x_3$-direction, the extended region is denoted by
\be
\mathbb{D}_e=\{(r, x_3):r_0<r<r_1, -3<x_3<3\}.
\ee
We extend $\psi_1, G_2$ as
\be\label{psicompat}
(\psi_1^e, G_2^e)(r, x_3)=\begin{cases}
(\psi_1, G_2)(r, x_3), \quad\quad\quad\,\,\quad (r, x_3)\in\bar{\mathbb{D}},\\
-(\psi_1, G_2)(r, -2-x_3), \quad (r, x_3)\in[r_0, r_1]\times[-3, -1],\\
-(\psi_1, G_2)(r, 2-x_3), \quad\,\,\,\,  (r, x_3)\in[r_0, r_1]\times[1, 3].
\end{cases}
\ee
Thus $\psi_1^e$ satisfies
\be\label{4111}
\begin{cases}
(\p_r^2+\p_{x_3}^2)\psi_1^e=G_2^e(W_1^\sharp, W_2, W_3^\sharp, W_4, W_5, W_6^\sharp)\in C^{1, \alpha}(\mathbb{D}^e),\\
\psi_1^e(r_1, x_3)=\p_r\psi_1^e(r_0, x_3)=0,\\
\psi_1^e(r, \pm 1)=0.
\end{cases}
\ee
Therefore, $\psi_1(r, x_3)\in C^{3, \al}(\bar{\mathbb{D}})$ and satisfies
\be\label{412}
\begin{cases}
||\psi_1||_{C^{3, \al}(\bar{\mathbb{D}})}\leq C||G_2(W_1^\sharp, W_2, W_3^\sharp, W_4, W_5, W_6^\sharp)||_{C^{1, \al}(\bar{\mathbb{D}})}\leq C\epsilon,\\
\p_r\psi_1(r, \pm 1)=0, \quad \p_{x_3}^2\psi_1(r, \pm1)=0, \quad \p_{x_3}^2\p_r\psi_1(r, \pm 1)=0,
\end{cases}
\ee
owing to the compatibility condition of $G_2$ in \eqref{Gcompat}.

Define $\hat W_1=W_1+\p_{x_3}\psi_1$ and $\hat W_3=W_3-\p_r\psi_1$, we have
\be\label{414}
\begin{cases}
\p_r\left(r\bar\rho_t(1-\bar M_{1t}^2)\hat W_1\right)+\p_{x_3}(r\bar\rho_t \hat W_3)+\p_r\left(\frac{r\bar\rho_t \bar U_{1t}}{\bar c_t^2}W_6\right)=\p_r \tilde G_1+\p_{x_3}\tilde G_3,\\
\p_r\hat W_3-\p_3\hat W_1=0,\\
r(\p_r^2+\frac1 r\p_r+\p_{x_3}^2)W_6+\frac{r\bar\rho_t\bar U_{1t}}{\bar c_t^2}\hat W_1-\frac{r\bar\rho_t}{\bar c_t^2}W_6=\tilde G_4,\\
\hat W_1(r_1, x_3)=\epsilon \tilde U_{1, ex}(x_3),\\
\hat W_3(r_0, x_3)=\epsilon \tilde U_{3, en}(x_3),\\
W_6(r_0, x_3)=\epsilon\tilde\Phi_{en}(x_3),\quad W_6(r_1, x_3)=\epsilon\tilde\Phi_{ex}(x_3), \quad \p_{x_3}W_6(r, \pm 1)=0,
\end{cases}
\ee
where
\begin{align*}
&\tilde{G}_1(W_1^\sharp, W_2, W_3^\sharp, W_4, W_5, W_6^\sharp)=r\bigg(G_1+\bar\rho_t(1-\bar M_{1t}^2)\p_{x_3}\psi_1\\
&\quad\quad\quad\quad\quad\quad\quad\quad\quad\quad\quad\quad\quad\quad+\bar\rho_t\bar M_{1t}\bar M_{2t}W_2+\frac{\bar\rho_t\bar U_{1t}}{(\ga-1)A_0}W_4-\frac{\bar\rho_t\bar U_{1t}}{\bar c_t^2}W_5\bigg),\\
&\tilde{G}_3(W_1^\sharp, W_2, W_3^\sharp, W_4, W_5, W_6^\sharp)=r\left(G_3-\bar\rho_t\p_{r}\psi_1\right),\\
&\tilde{G}_4(W_1^\sharp, W_2, W_3^\sharp, W_4, W_5, W_6^\sharp)=r\left(G_4+\frac{\bar\rho_t\bar U_{1t}}{\bar c_t^2}\p_{x_3}\psi_1-\frac{\bar\rho_t\bar U_{2t}}{\bar c_t^2}W_2-\frac{\bar\rho_t}{(\ga-1)A_0}W_4+\frac{\bar\rho_t}{\bar c_t^2}W_5\right).
\end{align*}

By the compatibility conditions \eqref{compati} and \eqref{412}, one has the compatibility conditions for $\hat{W}_1, \hat{W}_3, \tilde G_1, \tilde G_3$:
\be\label{GGcompat}
\begin{cases}
\p_{x_3}\hat W_1(r, \pm 1)=0,\quad \hat W_3(r, \pm 1)=\p_{x_3}^2\hat W_3(r, \pm 1)=0,\\
\p_{x_3}\tilde G_1(r, \pm 1)=0,\quad
\tilde G_3(r, \pm 1)=0.
\end{cases}
\ee
%Utilizing \eqref{GGcompat} and the symmetric extension technique as above, one can verify that $\hat W_1$, $\hat W_3$, $W_6$ solve \eqref{414} in $\bar{\mathbb{D}}$.
%the extension of $\tilde G_1, \tilde{\tilde G}_4$ can be done similarly as $G_1, G_4$, and $\tilde G_3$ is similarly as $G_3$, then

It follows from the second equation of \eqref{414} that there exists a unique potential function $\psi(r, x_3)$ such that $\hat W_1=\p_r\psi,  \hat W_3=\p_{x_3}\psi$, and $\psi, W_6$ should satisfy the following elliptic equations:
\begin{eqnarray}\label{415}
	\begin{cases}
	\p_{r}\left(r\bar\rho_t(1-\bar M_{1t}^2)\p_{r}\psi\right)+\p_{x_3}(r\bar\rho_t\p_{x_3}\psi)
+\p_{r}\left(\frac{r\bar\rho_t\bar U_{1t}}{\bar c_t^2}W_6\right)=\p_r\tilde G_1+\p_{x_3}\tilde G_3,\\
		r\left(\p_r^2+\frac1 r\p_r+\p_{x_3}^2\right)W_6+\frac{r\bar\rho_t\bar U_{1t}}{\bar c_t^2} \p_{r}\psi-\frac{r\bar\rho_t}{\bar c_t^2}W_6=\tilde G_4,
	\end{cases}
\end{eqnarray}
with the corresponding boundary conditions
\be\label{4155}
\begin{cases}
%(c^2(\rho_b)-U_{b1}^2)\p_r^2\psi+c^2(\rho_b)\p_3^2\psi+e_1(r)\p_r\psi+U_{b1}\p_rW_6=\tilde{G}_1(U_1^\sharp, U_2, U_3^\sharp, A, K, \Phi^\sharp),\\
%\Delta W_6-\frac{\rho_b}{c^2(\rho_b)}W_6+\frac{\rho_bU_{b1}}{c^2(\rho_b)}W_1=G_3(U_1^\sharp, U_2, U_3^\sharp, A, K, \Phi^\sharp),\\
\p_{x_3}\psi(r_0, x_3)=\epsilon \tilde U_{3, en}(x_3),\\
\p_r\psi(r_1, x_3)=\epsilon \tilde U_{1, ex}(x_3),\\
\p_r\psi(r, \pm 1)=0, \quad \psi(r_0, 0)=0,\\
\p_{x_3}\psi(r, \pm 1)=0,\\
W_6(r_0, x_3)=\epsilon\tilde\Phi_{en}(x_3),\quad W_6(r_1, x_3)=\epsilon\tilde\Phi_{ex}(x_3),\\
\p_{x_3}W_6(r, \pm1)=0.
\end{cases}
\ee

Thanks to the condition on the background solution, one has $1-\bar M_{1t}^2(r)>0$ for any $r\in(r_0, r_1)$. Let
\be
W_6=\phi+\phi_1,
\ee
where
\begin{eqnarray}
	\phi_1(r, x_3)=\epsilon\left(
	\frac{r_1-r}{r_1-r_0}\tilde\Phi_{en}(x_3)
	+\frac{r-r_0}{r_1-r_0}\tilde\Phi_{ex}(x_3)
	\right).
\end{eqnarray}
Then the system \eqref{415}-\eqref{4155} is transformed into
\begin{eqnarray}\label{416}
	\begin{cases}
	\p_{r}\left(r\bar\rho_t(1-\bar M_{1t}^2)\p_{r}\psi\right)+\p_{x_3}(r\bar\rho_t\p_{x_3}\psi)
+\p_{r}\left(\frac{r\bar\rho_t\bar U_{1t}}{\bar c_t^2}\phi\right)=\p_r\tilde{\tilde{G}}_1+\p_{x_3}\tilde{\tilde{G}}_3,\\
		r\left(\p_r^2+\frac1 r\p_r+\p_{x_3}^2\right)\phi+\frac{r\bar\rho_t\bar U_{1t}}{\bar c_t^2} \p_{r}\psi-\frac{r\bar\rho_t}{\bar c_t^2}\phi=\tilde{\tilde G}_4,\\
\p_r\psi(r_1, x_3)=\epsilon \tilde U_{1, ex}(x_3),\\
\p_{x_3}\psi(r_0, x_3)=\epsilon\tilde U_{3, en}(x_3),\\
%\p_r\psi(r_0, \pm 1)=0, \,\,\psi(r_0, 0)=0,\\
\p_{x_3}\psi(r, \pm 1)=0,\\
\phi(r_0, x_3)=\phi(r_1, x_3)=0,\\
\p_{x_3}\phi(r, \pm 1)=0,
%\p_r\psi(r, \pm 1)=\phi(r, \pm 1)=0,
	\end{cases}
\end{eqnarray}
where
\begin{align*}
&\tilde{\tilde{G}}_1(W_1^\sharp, W_2, W_3^\sharp, W_4, W_5, W_6^\sharp)=\tilde G_1-\frac{r\bar\rho_t\bar U_{1t}}{\bar c_t^2}\phi_1,\\
&\tilde{\tilde{G}}_3(W_1^\sharp, W_2, W_3^\sharp, W_4, W_5, W_6^\sharp)=\tilde G_3,\\
&\tilde{\tilde{G}}_4(W_1^\sharp, W_2, W_3^\sharp, W_4, W_5, W_6^\sharp)=\tilde G_4+\frac{r\bar\rho_t}{\bar c_t^2}\phi_1-r(\p_r^2+\frac1 r\p_r+\p_{x_3}^2)\phi_1.
\end{align*}
Introduce the function space $\mathscr{H}$ in $\bar{\mathbb{D}}$ by
\begin{equation*}
	\mathscr{H}=\{(\eta,\xi)\in H^1(\bar{\mathbb{D}})\times H^1(\bar{\mathbb{D}}):\int_{-1}^1\int_{r_0}^{r_1}\eta rdrdx_3=0,\quad\xi(r_0,x_3)=\xi(r_1, x_3)=0\}.
\end{equation*}
For $\psi, \phi\in\mathscr{H}$, direct calculations yield
\be\label{BL}
\mathcal{B}[(\psi, \phi), (\eta, \xi)]=\mathcal{L}(\eta, \xi),
\ee
where
\begin{align*}
\mathcal{B}[(\psi, \phi), (\eta, \xi)]=\int_{-1}^1\int_{r_0}^{r_1}&\Bigg(r\bar\rho_t(1-\bar M_{1t}^2)\p_r\psi\p_r\eta+r\bar\rho_t\p_{x_3}\psi\p_{x_3}\eta+\frac{r\bar\rho_t\bar U_{1t}}{\bar c_t^2}\phi\cdot\p_r\eta\\
&+r\p_r\phi\p_r\xi+r\p_{x_3}\phi\p_{x_3}\xi-\frac{r\bar\rho_t\bar U_{1t}}{\bar c_t^2}\p_r\psi\cdot\xi+\frac{r\bar\rho_t}{\bar c_t^2}\phi\xi\Bigg) drdx_3,
\end{align*}
and
\begin{align*}
\mathcal{L}(\eta, \xi)=&\int_{-1}^1\int_{r_0}^{r_1}(\tilde{\tilde{G}}_1\cdot\p_r\eta+\tilde{\tilde{G}}_3\cdot\p_{x_3}\eta-\tilde{\tilde{G}}_4\xi)drdx_3+\int_{-1}^1\eps r_1\bar\rho_t(1-\bar M_{1t}^2)\tilde U_{1, ex}\eta(r_1, x_3)dx_3\\
%&-\int_{-1}^1\eps\bar\rho(1-\bar M_{1}^2){\color{red}\p_r\psi(x_3)}\cdot\eta(r_0, x_3)dx_3\\
&-\int_{-1}^1\left(\tilde G_1\eta(r_1, x_3)r_1-\tilde G_1\eta(r_0, x_3)r_0\right)dx_3-\int_{r_0}^{r_1}\left(\tilde G_3\eta(r, 1)-\tilde G_3\eta(r, -1)\right)dr.
%&+\int_{r_0}{r_1}\rho_b\tilde U_{3, ex}(x_3)\eta(r, 1)dr-\int_{r_0}{r_1}\rho_b\p_{x_3}\psi\eta(r, -1)dr
\end{align*}

It should be noted that the special structure also exist here, that is, the mixed terms in $\mathcal{B}[(\psi,\phi),(\psi,\phi)]$ cancel each other, so that the coercivity of the bilinear operator $\mathcal{B}[(\psi,\phi),(\psi,\phi)]$ is valid,
\begin{align*}
\mathcal{B}[(\psi, \phi), (\psi, \phi)]&=\int_{\bar{\mathbb{D}}}\bigg(\bar\rho_t(1-\bar M_{1t}^2)(\p_r\psi)^2+\bar\rho_t(\p_{x_3}\psi)^2+(\p_r\phi)^2+(\p_{x_3}\phi)^2+\frac{\bar\rho_t}{\bar c_t^2}\phi^2\bigg) rdrdx_3,\\
%&\geq\int_{-1}^1\int_{r_0}^{r_1}\Bigg(\bar\rho_t(1-\bar M_{1t}^2)|\nabla\psi|^2+\bar\rho_t|\nabla\phi|^2+ \frac{\bar\rho_t}{\bar c_t^2}\phi^2\Bigg) rdrdx_3\\
&\geq C(||\psi||_{H^1(\bar{\mathbb{D}})}^2+||\phi||_{H_0^1(\bar{\mathbb{D}})}^2),
\end{align*}
for some positive constant $C$. Note that the boundedness of $\mathcal{L}(\eta, \xi)$ is also obvious. Thus, the Lax-Milgram theorem tells that there exists a unique $(\psi, \phi)\in\mathscr{H}$ to \eqref{416}.

By applying the regularity theory developed in \cite{Bae1}, we have $(\psi, W_6)\in (C^{2, \al}(\bar{\mathbb{D}}))^2$ with
\begin{eqnarray}\label{C1al}
	\begin{split}	
		\|(\psi, W_6)\|_{C^{2,\al}(\bar{\mathbb{D}})}
		\leq C
	\bigg(
	&\|\tilde G_1\|_{C^{1,\al}(\bar{\mathbb{D}})}+\|\tilde G_3\|_{C^{1,\al}(\bar{\mathbb{D}})}
	+\|\tilde G_4\|_{C^{\al}(\bar{\mathbb{D}})}+\|\tilde U_{1, ex}\|_{C^{1,\al}(\bar \Sigma^{1})}\\
&+\|\tilde U_{3, en}\|_{C^{1,\al}(\bar \Sigma^{0})}
	+\epsilon\|\tilde\Phi_{en}\|_{C^{2,\al}(\bar \Sigma^{0})}
	+\epsilon\|\tilde\Phi_{ex}\|_{C^{2,\al}(\bar \Sigma^{1})}
	\bigg).
	\end{split}
\end{eqnarray}
%Therefore, we obtain a $C^{1,\al}$ smooth solution
%\begin{equation*}
%	W_1=\p_{r}\psi+\p_{x_3}\psi_1,\,
%	W_3=\p_{x_3}\psi+\p_{r}\psi_1,\,
%	W_6=\phi+\phi_1,\,
%\end{equation*}
%to the system \eqref{49}.

Next, we consider the following linear system for $W_1, W_3$:
\be\label{499}
\begin{cases}
\p_r\left(r\bar\rho_t(1-\bar M_{1t}^2)W_1\right)+\p_{x_3}(r\bar\rho_t W_3)=-\p_r\left(\frac{r\bar\rho_t \bar U_{1t}}{\bar c_t^2}W_6\right)
+\p_r\left(r\bar\rho_t\bar M_{1t}\bar M_{2t}W_2\right)\\
\quad\quad\quad\quad\quad\quad\quad+\p_r\left(\frac{r\bar\rho_t\bar U_{1t}}{(\ga-1)A_0}W_4-\frac{r\bar\rho_t\bar U_{1t}}{\bar c_t^2}W_5\right)+\p_r(r G_1)+\p_{x_3}(r G_3),\\
\p_r W_3-\p_{x_3} W_1=G_2,\\
W_1(r_1, x_3)=\eps \tilde U_{1, ex}(x_3),\\
W_3(r_0, x_3)=\eps\tilde U_{3, en}(x_3),\\
W_3(r, \pm 1)=0, \quad \forall r\in[r_0, r_1].
\end{cases}
\ee

We extend ${\bf W}$ and $G_1, G_3, G_4$ as follows:
\be
(W_i^e, G_1^e, G_4^e)(r, x_3)=
\begin{cases}
(W_i^e, G_1, G_4)(r, x_3), \quad\quad\quad\,\, (r, x_3)\in\bar{\mathbb{D}},\\
(W_i^e, G_1, G_4)(r, -2-x_3), \quad (r, x_3)\in[r_0, r_1]\times[-3, -1],\\
(W_i^e, G_1, G_4)(r, 2-x_3), \quad\,\,\,\,  (r, x_3)\in[r_0, r_1]\times[1, 3],
\end{cases}
\ee
where $i=1, 2, 4, 5 ,6 $, and
\be
(W_3^e, G_3^e)(r, x_3)=
\begin{cases}
(W_3, G_3)(r, x_3), \quad\quad\quad\,\,\quad (r, x_3)\in\bar{\mathbb{D}},\\
-(W_3, G_3)(r, -2-x_3), \quad (r, x_3)\in[r_0, r_1]\times[-3, -1],\\
-(W_3, G_3)(r, 2-x_3), \quad\,\,\,\,  (r, x_3)\in[r_0, r_1]\times[1, 3].
\end{cases}
\ee
And $(\tilde U_{3, en}, \tilde\Phi_{en})$, $(\tilde U_{1, ex}, \tilde\Phi_{ex})$, $\tilde{b}$ can be extended similarly, which are still denoted by $(\tilde U_{3, en}, \tilde\Phi_{en})$, $(\tilde U_{1, ex}, \tilde\Phi_{ex})$, and $\tilde{b}$. Then $W_1^e, W_3^e$ satisfy:
\be\label{499}
\begin{cases}
\p_r\left(r\bar\rho_t(1-\bar M_{1t}^2)W_1^e\right)+\p_{x_3}(r\bar\rho_t W_3^e)=-\p_r\left(\frac{r\bar\rho_t \bar U_{1t}}{\bar c_t^2}W_6^e\right)+\p_r\left(r\bar\rho_t\bar M_{1t}\bar M_{2t}W_2^e\right)\\
\quad\quad\quad\quad\quad+\p_r\left(\frac{r\bar\rho_t\bar U_{1t}}{(\ga-1)A_0}W_4^e-\frac{r\bar\rho_t\bar U_{1t}}{\bar c_t^2}W_5^e\right)+\p_r(r G_1^e)+\p_{x_3}(r G_3^e),\\
\p_r W_3^e-\p_{x_3} W_1^e=G_2^e,\\
%\Delta W_6^e+\frac{\bar\rho_t\bar U_{1t}}{\bar c_t^2}W_1^e-\frac{\bar\rho_t}{\bar c_t^2}W_6^e=-\frac{\bar\rho_t\bar U_{2t}}{\bar c_t^2}W_2^e-\frac{\bar\rho_t}{(\ga-1)A_0}W_4^e+\frac{\bar\rho_t}{\bar c_t^2}W_5^e+G_4^e,\\
W_1^e(r_1, x_3)=\eps \tilde U_{1, ex}(x_3),\quad \forall x_3\in[-3, 3],\\
W_3^e(r_0, x_3)=\eps\tilde U_{3, en}(x_3),\quad \forall x_3\in[-3, 3].
%W_6^e(r_0, x_3)=\eps\tilde\Phi_{en}(x_3), \quad W_6^e(r_1, x_3)=\eps\tilde\Phi_{ex}(x_3),\quad \p_{x_3}W_6^e=0 \quad \forall x_3\in[-3, 3].
\end{cases}
\ee
By using the compatibility conditions \eqref{compati} and \eqref{Gcompat}, one has $G_1^e, G_3^e\in C^{2, \al}(\mathbb{D}^e)$, which further yields that
$(W_1, W_3)\in C^{2, \al}(\bar{\mathbb{D}})$ with the following estimate
\be\label{zuizhong}
||(W_1, W_3)||_{C^{2, \alpha}(\bar{\mathbb{D}})}\leq C\left(\sum\limits_{j=1}^3||G_j||_{C^{2, \alpha}(\bar{\mathbb{D}})}+\sum\limits_{j=4}^5||W_j||_{C^{2, \alpha}(\bar{\mathbb{D}})}+||\tilde{U}_{1, ex}||_{C^{1, \alpha}(\bar{\Sigma}^0)}+||\tilde{U}_{3, en}||_{C^{1, \alpha}(\bar{\Sigma}^1)}\right)
\ee
holding according to the standard Schauder estimates.

Combining \eqref{C1al} and \eqref{zuizhong}, for any fixed ${\bf W}^\sharp\in \mathcal{S}_{\delta_0}$, we have so far constructed a ${\bf W}\in C^{2,\al}(\bar{\mathbb{D}})$ with
\begin{equation}
	\sum_{j=1}^{6}\|W_j\|_{C^{2,\al}(\bar{\mathbb{D}})}\leq C_*^{'}(\epsilon+\delta_0^2+\epsilon\delta_0).
\end{equation}
Selecting $\delta_0=2 C_*^{'}\epsilon $ so that
\begin{equation}
	\|{\bf W}\|_{\mathcal{S}_{\delta_0}}\leq C_*^{'}\epsilon +C_*^{'}(2 C_*^{'}+1)\epsilon\delta_0<\delta_0,
\end{equation}
for any $\epsilon<\epsilon_*:=\frac{1}{2C_*^{'}(1+2C_*^{'})}$.

Now we turn to finish the proof of Theorem \ref{th3} by the contraction mapping principle. To accomplish this, we let
${\bf W}^{\sharp, k}\in\mathcal{S}_{\delta_0}, k=1,2$, and thus ${\bf W}^k=\mathcal{P}({\bf W}^{\sharp, k})$.
Define the deviations of every variables as
\begin{equation*}
	{\bf Y}^\sharp={\bf W}^{\sharp, 1}-{\bf W}^{\sharp, 2}, \,
	{\bf Y}={\bf W}^1-{\bf W}^2.
\end{equation*}

The estimates of $Y_4$ and $Y_5$ can be derived similarly as in Subsection \ref{subsec1}. And $Y_2$ is handled in the same way as $Y_4$ and $Y_5$. Next, it follows from \eqref{49} that
\be\label{332}
\begin{cases}
\p_r\left(r\bar\rho_t(1-\bar M_{1t}^2)Y_1\right)+\p_{x_3}(r\bar\rho_t Y_3)+\p_r\left(\frac{r\bar\rho_t \bar U_{1t}}{\bar c_t^2}Y_6\right)\\
\quad\quad=\p_r\left(r\bar\rho_t\bar M_{1t}\bar M_{2t}Y_2\right)+\p_r\left(\frac{r\bar\rho_t\bar U_{1t}}{(\ga-1)A_0}Y_4-\frac{r\bar\rho_t\bar U_{1t}}{\bar c_t^2}Y_5\right)\\
\quad\quad\quad+\p_r\left(r\left(G_1({\bf W}^{\sharp, 1})- G_1({\bf W}^{\sharp, 2})\right)\right)+\p_{x_3}\left(r\left(G_3({\bf W}^{\sharp, 1})-G_3({\bf W}^{\sharp, 2})\right)\right),\\
\p_r Y_3-\p_3 Y_1=G_2({\bf W}^{\sharp, 1})-G_2({\bf W}^{\sharp, 2}),\\
(\p_r^2+\frac1 r\p_r+\p_{x_3}^2)Y_6+\frac{\bar\rho_t\bar U_{1t}}{\bar c_t^2}Y_1-\frac{\bar\rho_t}{\bar c_t^2}Y_6=-\frac{\bar\rho_t\bar U_{2t}}{\bar c_t^2}W_2-\frac{\bar\rho_t}{(\ga-1)A_0}Y_4+\frac{\bar\rho_t}{\bar c_t^2}Y_5+G_4({\bf W}^{\sharp, 1})-G_4({\bf W}^{\sharp, 2}),\\
Y_1(r_1, x_3)=0,\\
Y_3(r_0, x_3)=0,\\
Y_3(r, \pm 1)=0,\\
Y_6(r_0, x_3)=0, \quad Y_6(r_1, x_3)=0,
%W_1(r, \pm 1)=W_6(r, \pm 1)=0.
\end{cases}
\ee
here $G_j({\bf W}^{\sharp, k}), j=1,2,3,4, k=1,2$ are functions defined in \eqref{49} by replacing ${\bf W}^{\sharp}$ by ${\bf W}^{\sharp, k}$.

Consequently, we have the following estimates
\begin{equation}\label{333}
\begin{split}
&\quad\|Y_1\|_{C^{1,\al}(\bar{\mathbb{D}})}+\|Y_3\|_{C^{1,\al}(\bar{\mathbb{D}})}
+\|Y_6\|_{C^{1,\al}(\bar{\mathbb{D}})}\\
&\leq C\left(
\sum_{j=1}^{4}\|G_j({\bf W}^{\sharp, 1})-G_j({\bf W}^{\sharp, 2})\|_{C^{1,\al}(\bar{\mathbb{D}})}
+\|(Y_4,Y_5)\|_{C^{1,\al}(\bar{\mathbb{D}})}\right)\leq C\delta_0\|{\bf Y}^\sharp\|_{C^{1,\al}(\bar{\mathbb{D}})}.
\end{split}
\end{equation}
Therefore, we obtain
\begin{equation}
	\|{\bf Y}\|_{C^{1,\al}(\bar{\mathbb{D}})}
	\leq C_{**}^{'} \delta_0
	\|{\bf Y}^\sharp\|_{C^{1,\al}(\bar{\mathbb{D}})}.
\end{equation}
The contraction of the mapping $\mathcal{P}$ in terms of $C^{1, \alpha}$ norm is verified by selecting $\delta_0$ as $C_{**}^{'} \delta_0\leq\frac2 3$. Then it turns out that there exists a unique fixed point ${\bf W}\in\mathcal{S}_{\delta_0}$ of $\mathcal{P}$. The proof of Theorem \ref{th3} is completed.

~\\\
\noindent{\bf \Large{Acknowledgement}.}

Weng is partially supported by National Natural Science Foundation of China 11971307, 12071359, 12221001.

\end{document}